      \documentclass[a4paper,fleqn]{article}
      \usepackage[centertags]{amsmath}
      \usepackage{amsthm, amssymb}
      \newcommand{\N}{N}
      \newcommand{\E}{E}
      \usepackage{fancyhdr}
      \newtheorem{thm}{Theorem}[section]
      \newtheorem{prop}[thm]{Proposition}
      \newtheorem{cor}[thm]{Corollary}
      \newtheorem{lem}[thm]{Lemma}

      \makeatletter                                           
      \def\th@newremark{\th@remark\thm@headfont{\bfseries}}   
      \makeatletter                                           

      \theoremstyle{definition}
      \newtheorem{definitionx}[thm]{Definition}

      \theoremstyle{newremark}
      \newtheorem{remarkx}[thm]{Remark}
      
      \newtheorem{examplex}[thm]{Example}

      \newenvironment{example}
	{\pushQED{\qed}\examplex}
	{\popQED\endexamplex}
      \newenvironment{rem}
	{\pushQED{\qed}\remarkx}
	{\popQED\endremarkx}

      \newenvironment{defn}
	{\pushQED{\qed}\definitionx}
	{\popQED\enddefinitionx}
\usepackage[utf8]{inputenc}
\usepackage[USenglish]{babel}
\usepackage{color}

 \usepackage{graphicx}
 \usepackage{pgf,tikz}

 \usepackage[pdftex,
            pdfauthor={M. Grieshammer, T. Rippl},
            pdftitle={Partial orders on metric measure spaces},
            pdfproducer={pdfLaTex},
            pdfcreator={}]{hyperref}

\newcommand{\eps}{\varepsilon}

\newcommand{\dx}{d}                                   
\newcommand{\dr}{\underline{\underline{r}}}                    
\DeclareMathOperator{\supp}{supp}                              
\DeclareMathOperator{\Poiss}{Poiss}                            
\DeclareMathOperator{\Exp}{Exp}                                

 \newcommand{\R}{\mathbb{R}}
 \renewcommand{\N}{\mathbb{N}}

 \renewcommand{\E}{\mathbb{E}}
 
\newcommand{\mcA}{\mathcal{A}}
\newcommand{\mcB}{\mathcal{B}}

\newcommand{\mcF}{\mathcal{F}}

\newcommand{\mcM}{\mathcal{M}}

\newcommand{\mcO}{\mathcal{O}}

\newcommand{\mfu}{\mathfrak{u}}
\newcommand{\mfU}{\mathfrak{U}}
\newcommand{\mfv}{\mathfrak{v}}

\newcommand{\mfw}{\mathfrak{w}}

\newcommand{\mfx}{\mathfrak{x}}
\newcommand{\mfX}{\mathfrak{X}}
\newcommand{\mfy}{\mathfrak{y}}
\newcommand{\mfY}{\mathfrak{Y}}
\newcommand{\mfz}{\mathfrak{z}}

\newcommand{\bbM}{\mathbb{M}}
\newcommand{\bbN}{\mathbb{N}}

\newcommand{\bbR}{\mathbb{R}}

\newcommand{\bbU}{\mathbb{U}}

\newcommand{\1}{1}
\newcommand{\U}{\mathbb{U}}

\DeclareMathOperator{\lemetric}{\leq_{\textup{metric}}}
\DeclareMathOperator{\lemeasure}{\leq_{\textup{measure}}}
\DeclareMathOperator{\legen}{\leq_{\textup{general}}}

\DeclareMathOperator{\LUB}{LUB}
\DeclareMathOperator{\ER}{ER}

\newenvironment{proofsteps}{\setcounter{enumi}{0}}{}
\newcommand{\step}{\refstepcounter{enumi}\removelastskip\smallskip\par\noindent\emph{Step \arabic{enumi}.} \hspace{0.5ex}}

\DeclareFontFamily{U}{mathx}{\hyphenchar\font45}
\DeclareFontShape{U}{mathx}{m}{n}{
      <5> <6> <7> <8> <9> <10>
      <10.95> <12> <14.4> <17.28> <20.74> <24.88>
      mathx10
      }{}
\DeclareSymbolFont{mathx}{U}{mathx}{m}{n}
\DeclareMathSymbol{\bigtimes}{1}{mathx}{"91}

\def\dEur{d_{\text{Eur}}}
\def\dgEur{d_{\text{gEur}}}

\numberwithin{equation}{section}  
                                 
\begin{document}

	  \title{Partial orders on metric measure spaces}
	  \author{Max Grieshammer\footnote{Institute for Mathematics, Friedrich-Alexander Universität Erlangen-Nürnberg, Germany;  	max.grieshammer@math.uni-erlangen.de, MG was supported by DFG SPP 1590}, Thomas Rippl\footnote{Institute for Mathematical Stochastics, Georg-August-Universit\"at G\"ottingen, G\"ottingen, Germany; trippl@uni-goettingen.de, TR was supported by DFG GR 876/15-1,2 of Andreas Greven}
	  }
	  
	  \maketitle
	  
	  \thispagestyle{empty}
	  \begin{abstract}
	  A partial order on the set of metric measure spaces is defined; it generalizes the Lipschitz order of Gromov.
	  We show that our partial order is closed when metric measure spaces are equipped with the Gromov-weak topology 
	  and give a new characterization for the Lipschitz order. \par 
	  We will then consider some probabilistic applications. The main importance is given to the study of Fleming-Viot processes with different resampling rates.
	  Besides that application we also consider tree-valued branching processes and two semigroups on metric measure spaces.
	  \end{abstract}

	  \noindent {\bf Keywords:} metric measure space, partial order,  tree-valued Fleming-Viot process

	  \smallskip

	  \noindent {\bf AMS 2010 Subject Classification:} Primary: 60E15, 53C23; Secondary: 60J25, 06A06.\\

%
%
%
%

\tableofcontents

\section{Introduction}

Stochastic order of random variables is particularly well-studied for random variables with values in the totally ordered space $\R$.
There are extensions to the partially ordered space $\R^d$, see \cite{shaked2007stochastic}.
Since recently the interest on random variables with values in metric spaces and metric measure spaces has grown (see \cite{EPW06} or \cite{gpw_mp}) we propose to study an order structure on metric (measure) spaces.

Thus, consider two metric spaces $(X,r_X)$ and $(Y,r_Y)$.
Is there a notion which can tell us that $(X,r_X)$ is \emph{smaller} than $(Y,r_Y)$?

We define such a notion; not on the set of metric spaces but on the set of metric measure spaces.
\footnote{For the ordering of metric spaces consider the introduction of \cite{EM14} which contains a perfect list of references.}
A metric measure space $(X,r,\mu)$ is a complete and separable metric space $(X,r)$ and a finite measure $\mu$ on (the Borel $\sigma$-field of) $X$. It is convenient to go to equivalence classes
$[X,r,\mu]$ of such metric measure spaces:
we say that a metric measure space $(X,r_X,\mu_X)$ is equivalent to a metric measure space $(Y,r_Y,\mu_Y)$ 
if we find a measure-preserving isometry $\supp(\mu_Y) \to \supp(\mu_X)$. We denote the set of such equivalence classes by $\bbM$ and write 
$\mfx = [X,r_X,\mu_X]$ and $\mfy = [Y,r_Y,\mu_Y]$ for typical elements $\mfx,\mfy \in \bbM$.
Metric measure spaces were studied in great detail in \cite{gromov} and \cite{sturm} as classical references and \cite{GPW09} and \cite{ALW14a} as probability theory related references.
One of the main reasons to prefer metric measure spaces to purely metric spaces for the ordering question are the powerful analytical tools of the former. \par 

In order to define a partial order $\legen$ on $\bbM$ we use the following two ideas: Compare masses and compare distances. I.e.\ we say 
$\mfx\legen \mfy$ if there is a Borel-measure $\mu_Y'$ on $Y$ such that $\mfx\lemetric [Y,r_Y,\mu_Y'] =:\mfy' \lemeasure \mfy$.
Here we say 
$\mfx \lemetric \mfy'$ if there is a measure-preserving sub-isometry (i.e.\ $1-$Lipschitz map) $\supp(\mu_Y')\to \supp(\mu_X)$ and we say
$\mfy' \lemeasure \mfy$ if $\mu_Y' \le \mu_Y$ (if one writes $[Y',r_Y',\mu_Y'] = \mfy'$ then this is equivalent to finding a sub-measure-preserving isometry
$\supp(\mu_Y) \to Y'$).
In easy words and leaving away details we say that a pony is smaller than a horse:
a pony is a contracted version of a horse with less weight. \par

Partial orders on metric measure spaces were already considered before.
In Section~3.$\frac{1}{2}$.15 of Gromov's book \cite{gromov} the Lipschitz order $\succ$ is defined.
There are some other articles who studied $\succ$ and we mention \cite{shioya} for a comprehensive overview.
This relation $\succ$ is identical to $\lemetric$.
So the relation $\legen$ is an \emph{extension} to $\succ$.
Moreover, we can prove the important facts for the relation $\legen$:
We show that $\legen$ is a \emph{partial order} on $\bbM$ and that $\legen$ is \emph{closed}, i.e.\ $\{(\mfx,\mfy) \in \bbM\times \bbM:\ \mfx \legen \mfy\}$ is closed in the product topology, where $\bbM$ is equipped with the Gromov-weak topology (see Definition~2.5 in \cite{ALW14a}).
Considering the partial order $\lemetric$ we provide an analytical characterization with distance matrix measures, see \eqref{rg6}.

In some cases partial orders on metric spaces $(E,r)$ are ``natural'' in the sense that the distance $r(x,y)$ for two elements $x,y \in E$ with $x \leq y$ can be expressed in a simple way.
An example of that phenomenon is the metric induced by the $1$-norm on the partially ordered space $\R^n$, $n\in \N$ with coordinate-wise ordering.
For the partial order $\legen$ we will find that it is natural if we endow $\bbM$ with the generalized Eurandom metric which is defined in \cite{MG}.

\smallskip

	There are two main applications of the partial order $\legen$.
	The first is the \emph{Cartesian semigroup} defined in \cite{EM14} and the second one is the \emph{concatenation semigroup} given in \cite{infdiv}.
	In the Cartesian semigroup any (normalized) metric measure space can be uniquely decomposed into prime factors.
	Defining that an element dominates another if its prime factors (counting multiplicity) are contained in the other we have a special instance of the $\lemetric$ situation.
	In the concatenation semigroup ultrametric measure spaces with a given upper bound for the diameter can be uniquely decomposed into prime factors.
	Defining that an element dominates another if its prime factors (counting multiplicity) are contained in the other we have a special instance of the $\lemeasure$ situation.

	This article will treat in particular probabilistic applications of the partial orders $\legen$, $\lemeasure$ and $\lemetric$ on $\bbM$.
	Lately representing the genealogy of a randomly evolving population by (ultra-) metric measure spaces has received growing interest, see \cite{gpw_mp} and descendant articles.
	The domination of genealogies (in some of the senses we defined) is a particularly interesting question as there are several situations where this is expected to occur in some way.
	The most popular among these cases is the tree-valued Fleming-Viot process with and without selection, see Theorem~5 of \cite{DGP12}.
	
	Here we give two main examples for a probabilistic application: the tree-valued Feller diffusion (Section~\ref{s.feller}) and in great more detail the tree-valued Fleming-Viot process (Section~\ref{sec.CompFV}).
	In particular it turns out that for two Fleming-Viot processes with different diffusivity $\gamma' > \gamma >0$ the Wasserstein distance of their Eurandom distance is given by  $\frac{1}{\gamma} - \frac{1}{\gamma'}$.
	But that is the difference of the expected genealogical distance of two individuals.
	We note that the coupling-results for Fleming-Viot processes are not new and can be proven using coalescent models. But on the level of trees, that have in general much more complexity than only pairwise-distances, the coupling result and the result on the distances of the random trees are new, as far as we know.

\smallskip

{\bf Outline:} In Section~\ref{s.mmspace} we give the definition of metric measure spaces and the Gromov-weak topology.
In Section~\ref{ss.legen} we present our main results on the relation $\legen$. \par 
In Section~\ref{s.further} we study the definition of $\legen$ in more details: In Section~\ref{ss.lemeasure}, the concept of smaller masses is defined and in subsection~\ref{ss.lemetric} we describe the concept of comparing distances. For the latter we characterize in Section~\ref{s.constr.lub} a set of ``least upper bounds''. Just before that we give the connections of the partial order $\legen$ to the generalized Eurandom distance. \par 
We use the above concepts to prove in Section~\ref{s.proofs} the main results. Finally we give in Section~\ref{s.applications} several probabilistic applications: The connection of the partial order to the Cartesian semigroup in~\ref{ss.EM14}, some consequences for the stochastic dominance and Wasserstein distance of random metric measure spaces (see Section~\ref{ss.stoch.dom}), an example concerning tree-valued Feller diffusions (see Section~\ref{s.feller}) and finally a result for tree-valued Fleming-Viot processes (see Section~\ref{ss.TVMM} and \ref{sec.CompFV}).

\section{Metric measure spaces}\label{s.mmspace}

\begin{defn}[Metric measure spaces]\label{d.mm}

  \begin{enumerate}
    \item We call $(X,r,\mu)$ a \emph{metric measure space} \emph{(mm space)} if
    \begin{itemize}
      \item $(X,r)$ is a complete separable metric space, where we assume that $X \subset \mathbb R$,
      \item $\mu$ is a finite measure on the Borel subsets of $X$. 
    \end{itemize}
    \item We define an equivalence relation on the collection of mm spaces as follows: Two mm spaces $(X,r_X,\mu_X)$ and $(Y,r_Y,\mu_Y)$ are \emph{equivalent} if and only if there exists a measurable map $\varphi:X\to Y$ such that
        \begin{align}
         \mu_Y & = \mu_X\circ \varphi^{-1}  \text{ and }\\
          r_X(x_1,x_2) &= r_Y(\varphi(x_1),\varphi(x_2))\quad \forall x_1,x_2\in \supp(\mu_X)
          \, ,
        \end{align}
  i.e.~$\varphi$ restricted to $\supp(\mu_X)$ is an isometry onto its image and  $\varphi$ is measure preserving.

        We denote the equivalence class of a mm space $(X,r_X,\mu)$ by $[X,r_X,\mu]$.
        \item We denote the collection of  equivalence classes of mm spaces by
        \begin{equation}
          \bbM:=\left\{[X,r,\mu]:(X,r,\mu)\text{ is mm space} \right\}
        \end{equation}
        The subset $\bbM_1 = \{ \mfx = [X,r,\mu] \in \bbM:\, \mu(X) = 1\}$ is the set of those mm spaces where $\mu$ is a probability measure.
  \end{enumerate}
\end{defn}

\begin{rem}\label{r.semigroup}
 The semigroup $([0,\infty),\cdot)$ of real multiplication \emph{acts} on $\bbM$ in two ways: for $a \in [0,\infty)$ and $\mfx =[X,r,\mu] \in \bbM$ we define
 \begin{align} 
   a \ast [X,r,\mu] &:= [X,ar,\mu] ,\\
   a\cdot [X,r,\mu] &:= [X,r,a\mu] .
 \end{align}
So, by $\ast$ we denote a multiplication of the \emph{metric} and by $\cdot$ we denote a multiplication of the \emph{measure}.
It is clear that $\ast$ can be restricted to $\bbM_1$ (to be precise: $\ast([0,\infty),\bbM_1) \subset \bbM_1$), whereas $\cdot$ cannot be restricted.
\end{rem}

\begin{defn}[Distance matrix measure]
For an mm space $\mfx=[X,r,\mu]\in\bbM$ and $m\geq2$ we define the \emph{distance matrix map of order $m$}
\begin{equation}\label{rg5}
  R^{m,(X,r)}:X^m\to\R^{\binom{m}{2}}\,,\quad (x_i)_{i=1,\dotsc,m}\mapsto (r(x_i,x_j))_{1\leq i<j\leq m}
\end{equation}
and the \emph{distance matrix measure of order $m$}
  \begin{align} \label{rg6}
  \nu^{m,\mfx}(\dx \underline{\underline{r}}) &:=\mu^{\otimes m}\circ (R^{m,(X,r)})^{-1}(d\underline{\underline{r}})\\
  &=\mu^{\otimes m}(\{(x_1,\dotsc,x_m)\in X^m:(r(x_i,x_j))_{1\leq i<j\leq m}\in\dx \underline{\underline{r}}\})\,. \nonumber
\end{align}
For $m=1$ we set $\nu^{1,\mfx} := \bar{\mfx} := \mu(X)$ the \emph{total mass}.
\end{defn}

The finite subtrees with $m$ leaves can be described by the following test functions.

\begin{defn}[Monomials]
\label{d.poly}

For $m\geq1$ and $\phi\in C_b(\R^{\binom{m}{2}})$ (the space of bounded continuous functions $\R^{\binom{m}{2}} \rightarrow \mathbb R$), define the \emph{monomial}
\begin{equation}\label{rg8}
  \Phi=\Phi^{m,\phi}:\bbM\to\bbR\,,\quad \mfu\mapsto \langle\phi,\nu^{m,\mfx}\rangle = \int_{\R^{\binom{m}{2}}} \nu^{m,\mfx}(\dx \dr)\, \phi(\dr)\,
\end{equation}
and write $\Pi$ for the set of monomials. 

For convenience, we abbreviate the nonnegative monomials $ \Pi_+ := \{\Phi^{m,\phi} \in \Pi:\, \phi \geq 0\} $.
The algebra generated by $\Pi$ is denoted by $\mcA(\Pi)$ and called the set of \emph{polynomials}.
\end{defn}

We next recall the topology given in Definition 2.5 of \cite{ALW14a}.

\begin{defn}[Gromov-weak-topology]\label{D.gromov}

  We say that a sequence $(\mfx_n)_{n\in\bbN}$ of elements from $\bbM$ converges to $\mfx\in\bbM$ in the \emph{Gromov-weak topology} if and only if
  \begin{equation}
    \Phi(\mfx_n) \to \Phi(\mfx)
  \end{equation}
  for any $\Phi \in \Pi$, defined in \eqref{rg8}.
  The topology is denoted by $\mcO_{\text{Gweak}}$.
\end{defn}
\begin{rem}
 The topology of Gromov-weak convergence is equivalent to the convergence of the distance measures and can be metricized by the Gromov-Prohorov metric $d_{\text{GPr}}$.

 The metric space $(\bbM, d_{\text{GPr}})$ is complete and separable, see Proposition~4.8 in \cite{ALW14a}.
\end{rem}



\section{The partial order \texorpdfstring{$\legen$}{le.general} on metric measure spaces}\label{ss.legen}

We define a relation $\legen$ on the set $\bbM$ of metric measure spaces. 
It will turn out that $\legen$ is a partial order with some additional properties. 

\begin{defn}[The relation $\legen$]\label{d.legen}
 For $\mfx, \mfy \in \bbM$ we define $\mfx \legen \mfy$ if for $\mfx = [X,r_X,\mu_X]$ and $\mfy =[Y,r_Y,\mu_Y]$
 there is a Borel-measure $\mu_Y' \le \mu_Y$ on $Y$ and a map $\tau : \supp(\mu_Y') \rightarrow \supp(\mu_X)$
 such that
 \begin{align}
 \mu_X &= \mu_Y' \circ \tau^{-1} ,\\
 r_X(\tau(y_1),\tau(y_2)) &\le r_Y(y_1,y_2)  \ \text{ for all } y_1,y_2 \in \supp(\mu_Y) .
 \end{align}
 We say that $\tau$ is a \emph{measure-preserving} mapping and a \emph{sub-isometry}.
\end{defn}

Of course one needs to verify that this definition does not depend on the particular representation of $\mfx$ and $\mfy$. But this can be easily seen
by definition - any other representative is measure-preserving isometric to the first one.
Besides it is worth comparing the previous definition to Definition~\ref{d.mm}.
Before we give an example we note that the above definition 
consists of two ideas, namely: 

\begin{defn}[The relation $\lemeasure$]\label{d.lemeasure}
 Let $\mfx = [X,r_X,\mu_X]$, $\mfy = [Y,r_Y,\mu_Y]\in \bbM$.
 We say that $\mfx \lemeasure \mfy$ if there is an isometry $\tau : \supp(\mu_Y) \rightarrow X$
 such that 
 \begin{equation}
  \label{e.legen.mu}  \mu_X \le \mu_Y\circ \tau^{-1}. 
 \end{equation}
 We say that $\tau$ is a \emph{sub-measure preserving} isometry. 
\end{defn}

And

\begin{defn}[The relation $\lemetric$]\label{d.lemetric}
 Let $\mfx = [X,r_X,\mu_X], \mfy = [Y,r_Y,\mu_Y],\in \bbM_1$.
 We say that $\mfx \lemetric   \mfy$ if there is a  map $\tau : \supp(\mu_Y) \rightarrow \supp(\mu_X)$
 such that $\mu_Y\circ \tau^{-1} = \mu_X$ and
 \begin{equation}
 \label{e.legen.r} r_Y(y_1,y_2) \geq r_X(\tau(y_1),\tau(y_2)) \ \text{ for all } y_1,y_2 \in \supp(\mu_Y) .
 \end{equation}
\end{defn}

As above these definitions do not depend on the representatives and we remark: 

\begin{rem}\label{r.equi.def}
 $\mfx \legen \mfy$ iff there is an mm space $\mfy'$ such that $\mfx \lemetric \mfy'$ $\lemeasure \mfy$, where we can extend the definition of $\lemetric$
 to mm-spaces with the same mass. 
\end{rem}

Let us now apply the definition in an example.
Even though it is trivial it illustrates the two important concepts: larger in distance and larger in mass.

\begin{example}\label{ex.1}
 \begin{enumerate}
  \item\label{ex.1.a} $\mfx_1 = [X,r_X,\mu_X]= [\{a,b\},r(a,b)=1,(\delta_a+\delta_b)/2]$ and $\mfy_1 = [Y,r_Y,\mu_Y]=  [\{c,d\},r(c,d)=2,(\delta_c+\delta_d)/2]$.
  Define $\tau_1: Y \to X$ via $\tau_1(c)=a$ and $\tau_1(d)=b$. Then we have
  \begin{equation}
   r_X(\tau_1(c),\tau_1(d)) = r_X(a,b) = 1 \leq 2 =r_Y(c,d) .
  \end{equation}
  So \eqref{e.legen.r} holds, i.e.\ $\mfx_1 \lemetric \mfy_1$. By Remark~\ref{r.equi.def} this implies $\mfx_1 \legen \mfy_1$.	
  \item\label{ex.1.b} $\mfx_2 = [X,r_X,\mu_X]= [\{e\},0,\delta_e]$ and $\mfy_2 = [Y,r_Y,\mu_Y]= [\{f\},0,2\delta_f]$. Then 
  Then $\tau_2: Y \to X$ via $\tau_2(f) =e$ satisfies $r_X(\tau_2(f),\tau_2(f)) = r_X(e,e) = 0  = r_Y(f,f)$ is an isometry and 
  $\delta_e = \delta_f \circ \tau_2^{-1}  \le 2 \delta_f\circ \tau_2^{-1} $. Thus $\mfx_2 \lemeasure \mfy_2$. Again, by Remark~\ref{r.equi.def} 
  this implies $\mfx_2 \legen \mfy_2$.
 \end{enumerate}
${}$  \\[-0.3cm]
 \noindent If we use the semigroup actions $\cdot$ and $\ast$ defined in Definition~\ref{d.mm} we can also write the two examples as 
 $\mfx_1 \legen 2 \ast \mfx_1 = \mfy_1$ and $\mfx_2 \legen 2 \cdot\mfx_2=\mfy_2$.
\end{example}

We include another example.
\begin{example}\label{ex.pol.gen}
 Let $\mfx = [\{1,2,4\},r(i,j)=|i-j|,  \delta_1 + \delta_2 + \delta_4 ]$ and $\mfy = [\{1,2,3,4\},r(i,j)=|i-j|,\sum_{i=1}^4 \delta_i ]$.
 Then we can not find a map $\tilde \tau: \{1,2,3,4\} \to \{1,2,4\}$ that is a sub-measure preserving sub-isometry. But we still have $\mfx \legen \mfy$.
\end{example}

We will now present some results for $\legen$. The first point is that $\legen$ defines a \emph{partial order} on $\bbM$, 
i.e.\ a reflexive, transitive and antisymmetric relation. The second point is, that $\legen$ is \emph{closed}, 
i.e.\ for $\mfx_n,\mfy_n,\mfx,\mfy \in \bbM$ with $\mfx_n \to \mfx$ and $\mfy_n \to \mfy$ as $n\to \infty$ the following holds: $\mfx_n \legen \mfy_n$ for 
all $n \in \N$ implies that $\mfx \legen \mfy$.

\begin{thm}\label{t.legen.pots}
$\legen$ is a closed partial order on $\bbM$.
\end{thm}

\begin{rem}
 We could also define a partial order $\le'$ on $\bbM$, where we say $\mfx \le' \mfy$ if there is a sub-measure preserving sub-isometry $\supp(\mu_Y) \to \supp(\mu_X)$. 
 It is easy to see that $\mfx \le' \mfy$ implies $\mfx \legen \mfy$ but a slight modification of Example~\ref{ex.pol.gen} shows that this partial order is not closed. 
\end{rem}

The following result will be important for applications: 

\begin{prop}\label{p.legen.compact}
 Let $A \subset \bbM$ be compact. Then the set $\bigcup_{\mfy \in A} \{\mfx \in \bbM:\, \mfx \legen \mfy\}$ is compact.
\end{prop}

In some cases partially ordered sets have a deeper algebraic structure underlying which may come from a lattice.
In our case, however, there is no such structure, since $(\bbM,\legen)$ is neither a join-semilattice nor a meet-semilattice in general (the point is that we can not expect uniqueness of a ``greatest lower bound'' or ``least upper bound''). \par 
But we have the following properties  with respect to the semigroup actions given in Definition \ref{d.mm}. Namely we get that the partial order is compatible with the semigroup actions:

\begin{prop}\label{p.legen.easy}
 Let $\mfx, \mfy \in \bbM$.
 \begin{enumerate}
  \item\label{p.legen.easy.1} $a\cdot \mfx \legen \mfx$ for $a \in [0,1]$ and $\mfx \legen b \cdot \mfx$ for $b \in [1,\infty)$.
  \item\label{p.legen.easy.2} $a \ast \mfx \legen \mfx$ for $a \in [0,1]$ and $\mfx \legen b \ast \mfx$ for $b \in [1,\infty)$.
  \item\label{p.legen.easy.3} $\mfx \legen \mfy$ implies that $c \cdot \mfx \legen c \cdot \mfy$ and $c \ast \mfx \leq c \ast \mfy$ for any $c \in [0,\infty)$.
 \end{enumerate}
In particular, the first statement states that $0 = [\{a\},r,0] \legen \mfx$ for all $\mfx \in \bbM$.
\end{prop}

\section{Further results}\label{s.further}

Here we study the two special cases of the $\legen$ order, given in Definition~\ref{d.lemeasure} and Definition~\ref{d.lemetric}, in more details. Moreover we prove a connection to the Eurandom distance and define a set of least upper bounds. 

\subsection{The partial order \texorpdfstring{$\lemeasure$}{le.measure}}\label{ss.lemeasure}

In this section, we will describe the relation $\lemeasure$ given in Definition~\ref{d.lemeasure}  in more details. We start with the following observation:

\begin{prop} \label{p.lemeasure.pots}
 The relation $\lemeasure$ of Definition~\ref{d.lemeasure} is a closed partial order on $\bbM$.
\end{prop}

\begin{proof}
 Note that $\mfx \lemeasure \mfy \lemeasure \mfz = [Z,r_Z,\mu_Z] \in \bbM$ iff there are Borel-measures $\mu_X,\mu_Y$ on the Borel subsets of $Z$
 such that $\mfx = [Z,r_Z,\mu_X]$, $\mfy = [Z,r_Z,\mu_Y]$ and $\mu_X \le \mu_Y \le \mu_Z$ (with the classical partial order on measures). This implies that $\lemeasure$ is a partial order. \par 
 If we take $\mfx_n,\mfy_n,\mfx,\mfy \in \bbM$, $n \in \mathbb N$ with $\mfx_n \rightarrow \mfx$, $\mfy_n \rightarrow \mfy$ and 
 $\mfx_n \lemeasure \mfy_n$ for all $n \in \mathbb N$, then we need to show $\mfx \lemeasure \mfy$. Note that, as before, we find measures 
 $\mu_X^n \le \mu_Y^n$ such that  $\mfx_n = [Y^n,r_Y^n,\mu_X^n]$ and $\mfy_n = [Y^n,r_Y^n,\mu_Y^n]$, $n \in \mathbb N$. \par 
 By Lemma 5.8 in \cite{GPW09}, there is a complete separable metric space $(Z,r_Z)$ and 
 isometric embeddings $\varphi,\varphi_1,\varphi_2,\ldots$ from  $Y,Y^1,Y^2,\ldots $ into $(Z,r_Z)$ such that 
 \begin{equation}\label{e.Pr.1}
  d_{Pr}(\mu_Y^n \circ\varphi_n^{-1},\mu_Y\circ \varphi^{-1}) \rightarrow 0, 
 \end{equation}
where the Prohorov metric is defined on the set of Borel-measures on $(Z,r_Z)$. By the continuous mapping theorem we also know that 
 \begin{equation}\label{e.Pr.2}
  d_{Pr}(\mu_X^n \circ\varphi_n^{-1},\mu_X\circ \varphi^{-1}) \rightarrow 0. 
 \end{equation}
 Since $\mu_X^n \circ\varphi_n^{-1} \le \mu_Y^n \circ\varphi_n^{-1}$ for all $n \in \mathbb N$ we can combine  that with \eqref{e.Pr.1} and \eqref{e.Pr.2} to 
 $\mu_X\circ \varphi^{-1} \le \mu_Y\circ \varphi^{-1}$, and hence $\mfx = [Z,r_Z,\mu_X\circ \varphi^{-1}] \lemeasure [Z,r_Z,\mu_Y\circ \varphi^{-1}] = \mfy$.
\end{proof}

Let us relate the partially ordered set to a semigroup.

\begin{rem}\label{r.concat}
 The semigroup of concatenation is defined in \cite{infdiv}.
 Fix $h>0$ and define $\bbU(h)^\sqcup := \{\mfu \in \bbU: \, \nu^{2,\mfu}((h,\infty)) = 0\}$ as the space of $h$-forests.
 Those are the ultrametric measure spaces with distance at most $h$; they correspond to trees with height at most $h/2$, see the above reference for details.
 This space can be made a semigroup via the binary operation $\sqcup: \bbU(h)^\sqcup \times \bbU(h)^\sqcup \to \bbU(h)^\sqcup$, called $h$-concatenation:
 \begin{equation}\label{eq.con.1}
  [U_1,r_1,\mu_1] \sqcup [U_2,r_2,\mu_2] = [U_1 \uplus U_2, r_1 \sqcup r_2, \mu_1+\mu_2] \, ,
 \end{equation}
with $\uplus$ is the disjoint union and
 \begin{equation}
	\begin{split}
   r_1 \sqcup r_2 (x,y) = r_1(x,y) &\1(x,y \in U_1) + r_2(x,y) \1(x,y\in U_2) \\
	&+ h \1(x\in U_1,y \in U_2) + h \1(x\in U_2, y \in U_1) 
	\end{split}
 \end{equation}
 for $[U_1,r_1,\mu_1],[U_2,r_2,\mu_2] \in \U(h)^\sqcup$.
 As this turns out to be a cancellative operation, the induced relation $\leq_\sqcup$ 
 ($\mfu \leq_\sqcup \mfv :\Leftrightarrow \exists \mfw: \mfu \sqcup \mfw = \mfv$) defines a partial order. \par 
If now $\mfu= [U,r_U,\mu_U] \leq_\sqcup \mfv= [V,r_V,\mu_V] $ then $V$ is of the form (\ref{eq.con.1}), i.e.\ there is a $\mfw = [W,r_W,\mu_W]$ such that 
$[V,r_V,\mu_V] = [U \uplus W, r_U \sqcup r_W, \mu_U+\mu_W]$. Since $\mu_U + \mu_W \ge \mu_U$ (as measures on $U \uplus W$) and $\mfu = [U \uplus W, r_U \sqcup r_W, \mu_U]$ we get $\mfu \lemeasure\mfv$.
\end{rem}

We close this section with the following properties of $\lemeasure$:

\begin{prop}\label{p.prop.measure}
Let $\mfx,\mfy \in \bbM$.
\begin{enumerate}
\item\label{p.prop.measure.a} If $\mfx \lemeasure \mfy$ and $\overline{\mfx} = \overline{\mfy}$, then $\mfx = \mfy$.  
\item\label{p.lemeasure.compact}  Let $A \subset \bbM$ be compact.  Then the set $\bigcup_{\mfy \in A} \{\mfx \in \bbM:\, \mfx \lemeasure \mfy\}$ is compact. 
\end{enumerate}
\end{prop}

\begin{proof}
 (\ref{p.prop.measure.a}) Note that if two measures $\mu,\nu$ on a set $Y$ satisfy $\mu \le \nu$ and $ \mu(Y) = \nu(Y)$, this is enough 
 to get $\mu = \nu$. \par 
 (\ref{p.lemeasure.compact}) Take a sequence $(\mfx_n)_{n \in \mathbb N}$ in $\bigcup_{\mfy \in A} \{\mfx \in \bbM:\, \mfx \lemeasure \mfy\}$. Then there is a sequence 
 $(\mfy_n)_{n \in \mathbb N}$ in $A$ such that $\mfx_n \lemeasure \mfy_n = [Y^n,r_Y^n,\mu_Y^n]$ for all $n \in \mathbb N$. Since $A$ is compact we get $\mfy_n \rightarrow \mfy \in A$ along 
 some subsequence, where we suppress the dependence. Following the proof of Proposition~\ref{p.lemeasure.pots}, we find a complete separable metric space $(Z,r_Z)$ 
 and isometric embeddings $\varphi,\varphi_1,\varphi_2,\ldots$ from  $Y,Y^1,Y^2,\ldots $ into $(Z,r_Z)$ such that $\mu_Y^n \circ \varphi_n^{-1} \Rightarrow \mu_Y \circ 
 \varphi^{-1}$. Since $\mu_X^n \circ \varphi_n^{-1} \le \mu_Y^n \circ \varphi_n^{-1} $ for all $n \in \mathbb N$, Prohorov's theorem implies 
 $\mu_X^n \circ \varphi_n^{-1} \Rightarrow \mu_X$ along some subsequence, where we again suppress the dependence. With the same argument as after \eqref{e.Pr.2} we get   $\mu_X \le \mu_Y \circ \varphi^{-1}$ and by Lemma~5.8 in \cite{GPW09}, this
 is enough to prove $\mfx_n \rightarrow [Z,r_Z,\mu_X] =:\mfx$. Since $\mfx \lemeasure [Z,r_Z,\mu_Y \circ \varphi^{-1}] = [Y,r_Y,\mu_Y]$, the result follows. 
\end{proof}

\subsection{The partial order \texorpdfstring{$\lemetric$}{le.metric}}\label{ss.lemetric}

In this section, we  describe the relation $\lemetric$ given in  Definition~\ref{d.lemetric} in more details. 
It will turn out that it is a closed partial order. Before we start we note that although $\lemetric$ is a relation on $\bbM_1$, 
it can be extended without any problems to compare mm-spaces with the same mass. \\ 

Our first result on the relation $\lemetric$ is a characterization in terms of monomials introduced in Definition~\ref{d.poly}.
Let $m \in \{2,3,\dotsc\} $ and define a partial order on $\R^{\binom{m}{2}}$: for the two elements $\dr , \dr' \in \R^{\binom{m}{2}}$ 
set $\dr \leq \dr'$ iff $\dr_{ij} \leq \dr_{ij}'$ for $1\leq i < j \leq m$.
Then we call a function $\phi \in C(\R^{\binom{m}{2}})$ \emph{increasing} if $\phi(\dr) \leq \phi(\dr')$ for all $\dr,\dr' \in\R^{\binom{m}{2}}$ 
with $\dr \leq \dr'$.
A set $A \subset \R^{\binom{m}{2}}$ is called \emph{increasing} if its indicator function $\1_A$ is increasing, i.e.~if $\dr \in A$ then $\dr' \in A$ 
for all $\dr' \in \R^{\binom{m}{2}}$ with $\dr \leq \dr'$.
A monomial $\Phi^{m,\phi} \in \Pi$ is called \emph{increasing} if $\phi$ is increasing.

\begin{thm}\label{p.lemetric.char} 
 Let $\mfx= [X,r_X,\mu_X], \mfy = [Y,r_Y,\mu_Y] \in \bbM_1$.
 The following are equivalent:
 \begin{enumerate}
  \item\label{t.lemetric.char.1} $\mfx \lemetric \mfy$.
  \item\label{t.lemetric.char.3} $\Phi(\mfx) \leq \Phi(\mfy)$ for all increasing $\Phi$.
  \item\label{t.lemetric.char.4} $\nu^{m,\mfx} (A) \leq \nu^{m,\mfy}(A)$ for all increasing $A \in  \mathcal B(\R^{\binom{m}{2}})$, $m\in  \N_{\ge 2}$.
  \item\label{t.lemetric.char.5} $\nu^{\infty,\mfx} (A) \leq \nu^{\infty,\mfy}(A)$ for all increasing 
  $A \in  \mathcal B(\R^{\binom{\mathbb N}{2}})$, where $\nu^{\infty,\mfx}$ is defined as in (\ref{rg6}) with $m$ replaced by $\infty$.
 \end{enumerate}
\end{thm}

\begin{rem} 
\begin{enumerate}
 \item The Theorem may be seen as an extension of Lemma~4.2 (b) of \cite{EM14}. Their work also defines a partial order and we will see later in Section~\ref{ss.EM14} that their partial order is a special case of our order.
 \item We think that this theorem is also true for the general order $\legen$, where one has to use positive increasing functions. 
But this is still open.
\end{enumerate}
\end{rem}

\begin{proof}
 ``(\ref{t.lemetric.char.1}) $\Rightarrow$ (\ref{t.lemetric.char.3})'' is straight forward and 
 ``(\ref{t.lemetric.char.3}) $\Rightarrow$ (\ref{t.lemetric.char.4})'' follows by a standard approximation argument. \par 
 For ``(\ref{t.lemetric.char.4}) $\Rightarrow$ (\ref{t.lemetric.char.5})''  we note that 
 $\mathbb R^{\binom{\mathbb N}{2}} \supset A  = \bigcap_m \pi_m^{-1}(\pi_m(A))$, where 
 $\pi_m:\mathbb R^{\binom{\mathbb N}{2}} \to \mathbb R^{\binom{m}{2}} $ is the projection, and that 
 $\nu^{\infty,\mfx}(\pi_m^{-1}(\pi_m(A))) = \nu^{m,\mfx}(\pi_m(A))$. \par 
The proof of ``(\ref{t.lemetric.char.5}) $\Rightarrow$ (\ref{t.lemetric.char.1})'' is based on the proof of the mm-reconstruction Theorem (see for example \cite{Kondo} 
and \cite{Vershik}).
  We can assume w.l.o.g.~that $X = \supp(\mu_X)$ and $Y = \supp(\mu_Y)$. Let $E_X \subset X^\N$ be the set of all sequences $(x_i)_{i \in \N}$ with
  \begin{equation}
  \lim_{n \rightarrow \infty} \frac{1}{n} \sum_{i = 1}^n f(x_i) = \int_X f(x) \mu_X(dx),\qquad \forall f \in C_b(X).
  \end{equation}
  Note that $\mu_X^{\otimes N}(E_X) = 1$ (by the Glivenko-Cantelli theorem, e.g.~in \cite{parthasarathy}) and that $\{x_i: \ i \in \N\}$ is dense in $X$ 
  for all $x \in E_X$ (we assumed $X = \supp(\mu_X)$). We denote by $E_Y$ the analogue set of sequences in $Y$, where we replace $\mu_X$ by $\mu_Y$.
  Define 
  \begin{align}
  A&:= \left\{\dr \in \R_+^{\binom{\mathbb N}{2}}:\ \exists (x_i)_{i \in \N} \in E_X : r_X(x_i,x_j) \leq r_{i,j} ,\ \forall 1 \le i \le j \right\},\\
  B&:= \left\{\dr \in \R_+^{\binom{\mathbb N}{2}}:\ \exists  (y_i)_{i \in \N}\in E_Y:  \ r_Y(y_i,y_j) \geq r_{i,j},\  \forall 1 \le i \le j \right\}.
  \end{align}
  Clearly 
\begin{equation}
 \nu^{\infty,\mfx}(A) =  \nu^{\infty,\mfy}(B) = 1.
\end{equation}
Observe that $\R_+^{\binom{\mathbb N}{2}} \backslash B$ is an increasing set and we have
\begin{equation}
 \nu^{\infty,\mfx}\left(\R_+^{\binom{\mathbb N}{2}} \backslash B\right) \le \nu^{\infty,\mfy}\left(\R_+^{\binom{\mathbb N}{2}} \backslash B\right) = 0.
\end{equation}
It follows that $\nu^{\infty,\mfx}(A \cap B) = 1$ and hence $A \cap B$ is not empty. Now, by definition, we find a sequence $(x_i)_{i\in \N}\in E_X$ and $(y_i)_{i \in \N} \in E_Y$ with the property that  $r_X(x_i,x_j) \leq \dr_{ij} \leq r_Y(y_i,y_j)$ for all $i,j \in \N$. 
Fix these two sequences.
Define the map $\tilde\tau: \{y_i: i \in \N\} \rightarrow X$, $y_i \mapsto x_i$, then $\tilde\tau$ is a sub-isometry defined on a dense subset of $Y$ and therefore extends to a sub-isometry $\tau: Y \rightarrow X$.
Finally observe that by definition of the sequences $(x_i)_{i \in \N}$ and $(y_i)_{i \in \N}$:
  \begin{equation}
  \begin{split}
 \int f \ d\mu_Y\circ \tau^{-1} &= \int f\circ \tau d \mu_Y =  \lim_{n \rightarrow \infty} \frac{1}{n} \sum_{i = 1}^n f(\tau(y_i)) \\
  &=  \lim_{n \rightarrow \infty} \frac{1}{n} \sum_{i = 1}^n f(x_i) =   \int f  \ d \mu_X.
  \end{split}
  \end{equation}
  for all functions $f \in C_b(X)$, i.e.\ $\mu_Y\circ \tau^{-1} = \mu_X$ and therefore $\tau$ is a measure-preserving sub-isometry as required.
\end{proof}

As a direct consequence of Theorem~\ref{p.lemetric.char}, we can deduce the following known statement (see  3.$\frac{1}{2}$.15 (a) and (b) in \cite{gromov}).

\begin{prop}\label{p.lemetric.pots}
The relation $\lemetric$ of Definition~\ref{d.lemetric} is a closed partial order on $\bbM_1$.
\end{prop}

\begin{proof}
This proof follows directly from Proposition~\ref{p.lemetric.char}: While the reflexivity and transitivity are obvious, the antisymmetry follows 
by the fact that $\Phi(\mfx) = \Phi(\mfy)$ for all increasing $\Phi \in \Pi_+$ implies $\mfx = \mfy$. This follows since the algebra generated by
increasing $\Phi$ is dense in the set of all polynomials and this suffices to deduce  $\mfx = \mfy$ (see Proposition~2.6 in \cite{GPW09}). \par 
The closedness follows since the monomials generate the Gromov-weak topology.
\end{proof}

One may think that for ``small'' spaces (with few points) one only needs to look at low order polynomials.
The next example shows that this is not the case. Nevertheless we think that the characterization result, Theorem \ref{p.lemetric.char},
might be helpful algorithmically to determine whether $\mfx \lemetric \mfy$ holds.

\begin{example}
We consider  $\mfx = (\{a,b\},r(a,b)=1,(\delta_a + \delta_b)/2)$ and $\mfy = (\{c,d,e\},r(c,d)=1,r(c,e)=r(d,e)=2,(\delta_c+\delta_d + \delta_e)/3)$. Then, on the one hand, one can not find a measure preserving sub-isometry but on the other hand it is not obvious that the distance matrix distributions do not dominate each other. In particular one needs to consider the distance matrix distribution of order $m = 10$ to see that $\nu^{m,\mfx}\not \leq \nu^{m,\mfy}$: If we look at the sequence of points
\begin{equation}
\underline{x} := \left(\underbrace{a,\ldots,a }_{m},\underbrace{b,\ldots,b }_{m}\right)
\end{equation}
and denote by $R:= R^{m,\mfx}(\underline{x})$ the corresponding distance matrix, then 
\begin{equation}
\nu^{m,\mfx}\left(\bigtimes_{1 \le i < j \le m}[R_{i,j},\infty) \right) = \frac{2}{2^{2m}}. 
\end{equation}
On the other hand: 
\begin{equation}
\nu^{m,\mfy}\left(\bigtimes_{1 \le i < j \le m}[R_{i,j},\infty) \right) = \frac{3 \cdot 2^{m}+3\cdot (2^m-2)}{3^{2m}}. 
\end{equation}
It follows that
\begin{equation}
\begin{split}
\nu^{m,\mfy}&\left(\bigtimes_{1 \le i < j \le m}[R_{i,j},\infty) \right) \le \nu^{m,\mfx}\left(\bigtimes_{1 \le i < j \le m}[R_{i,j},\infty) \right) \\
& \Longleftrightarrow \qquad  2^{m+1}-2 \le \left(\frac{3}{2}\right)^{2m - 1}\\
& \Longleftrightarrow \qquad  m \ge 10. 
\end{split} 
\end{equation} 
So in this example to distinguish if a space of two points is dominated by one of three points one needs to consider the distance matrix distribution of order 10.
We do not know if one may formulate an upper bound on the necessary order depending on the number of points.
\end{example}

We close this section with some properties of $\lemetric$. 
\begin{prop}\label{p.prop.metric}
Let $\mfx,\mfy \in \bbM_1$. 
\begin{enumerate}
\item\label{p.prop.metric.a} If $\mfx \lemetric \mfy$ and $\nu^{2,\mfx} = \nu^{2,\mfy}$, then $\mfx = \mfy$.
\item\label{p.lemetric.compact} Let $A \subset \bbM_1$ be compact. Then the set $\bigcup_{\mfy \in A} \{\mfx \in \bbM_1:\, \mfx \lemetric \mfy\}$ is compact.
\item\label{p.prop.metric.c} There is a set  $\LUB(\mfx_1,\mfx_2)  \subset \bbM_1$, with the property: If $ \mfw \in \bbM_1$  with $\mfw \legen \mfx_1$ and $\mfw \legen \mfx_2$ then $\bar \mfz \le \mfw$ for some $\bar \mfz \in \LUB(\mfx_1,\mfx_2) $ implies $\bar \mfz = \mfw$.
\end{enumerate}
\end{prop}

We note that (\ref{p.prop.metric.c}) can be deduced by Zorn's Lemma. But in contrast to the other partial orders, we can characterize $\LUB(\mfx_1,\mfx_2)$ in this situation explicitly using optimal couplings for the involved measures. We will study the set $\LUB$ in Section~\ref{s.constr.lub}.

\begin{proof}
(\ref{p.prop.metric.a}) is Lemma 2.6 in \cite{shioya}. But we note that \eqref{p.prop.measure.a} is also a direct consequence of Theorem \ref{t.gen.Eur.order}. \par 
(\ref{p.lemetric.compact}) This is 3.$\frac{1}{2}$.15(c) in \cite{gromov}, but for completeness we will give a proof. Set 
\begin{equation}
 L(A) = \bigcup_{\mfy \in A} \{\mfx \in \bbM_1:\, \mfx \lemetric \mfy\}. 
\end{equation}
According to Proposition~7.1 in \cite{GPW09}, the set $L(A)$ is compact (note that it is closed) if:
\begin{align}
\label{e.nu2} &\left\{ \nu^{2,\mfy};\, \mfy \in L(A) \right\} \subset \mcM_1([0,\infty))\text{ is tight}\\
\label{e.tr54}&\sup_{\mfx \in L(A)} v_\delta(\mfx) \xrightarrow{\delta \to 0} \, 0,
\end{align}
where for $\mfx = [X,r,\mu]$, $\delta, \eps >0$ and $B_{\eps}^r(y) := \{x \in X: \, r(x,y)<\eps\}$
\begin{equation} \label{modul}
 v_{\delta}(\mfx) =  \inf \left\{ \eps:\, \mu\left\{x\in X:\,  \mu(B_{\eps}^r(x))\leq \delta \right\} < \eps \right\}  \, .
\end{equation}
By Theorem~\ref{p.lemetric.char}, (\ref{e.nu2}) is straight forward, since $A$ is compact (see again Proposition 7.1 in \cite{GPW09}). \par 
To prove (\ref{e.tr54}) we take $\mfx \in L(A)$. Then we find a $\mfy \in A$ such that $\mfx = [X,r_X,\mu_X]$ $\lemetric \mfy = [Y,r_Y,\mu_Y]$. 
This implies the existence of a measure-preserving sub-isometry $\tau: \supp(\mu_Y) \to \supp(\mu_X)$. It follows that 
$B^{r_Y}_\eps(y) \subset \tau^{-1}\big(B^{r_X}_\eps(\tau(y))\big)$, for $y \in \supp(\mu_Y)$ and hence
\begin{equation}
 \begin{split}
  v_{\delta}(\mfx)&= \inf \left\{ \eps:\, \mu_X\left\{x\in X:\,  \mu_X(B_{\eps}^{r_X}(x))\leq \delta \right\} < \eps \right\} \\
  &= \inf \left\{ \eps:\, \mu_Y\left\{x\in X:\,  \tau^{-1}B^{r_X}_\eps(\tau(y))\leq \delta \right\} < \eps \right\} \\
  &\le v_{\delta}(\mfy).
 \end{split}
\end{equation}
Combining this with the fact that $A$ is compact, (\ref{e.tr54}) follows again by Proposition~7.1 in \cite{GPW09}.\par 

For (\ref{p.prop.metric.c}) see Section~\ref{s.constr.lub}.
\end{proof}

\subsection{The generalized Eurandom distance}\label{s.gen.Eur}

The Eurandom-distance was introduced in \cite{GPW09}, Section~10 and is generalized in \cite{MG}. We recall the definition and 
 some of the results. For details we refer to the mentioned papers. \par 

Let $\mfx = [X,r_X,\mu_X],  \mfy = [Y,r_Y,\mu_Y] \in \mathbb M_1$ and $\lambda > 0$ then the (modified) Eurandom-metric is given by: 
\begin{equation}\label{e.def.Eurandom}
\begin{split}
\dEur^\lambda &(\mfx,\mfy):= \\
&\inf_{\tilde \mu \in \Pi(\mu_X,\mu_Y)} \int_{(X\times Y)^2} \left|e^{-\lambda r_Y(y,y')}-e^{-\lambda r_X(x,x')}\right| \tilde \mu(d(x,y))\tilde \mu(d(x',y')),
\end{split}
\end{equation}
where the infimum is taken over all couplings $\Pi(\mu_X,\mu_Y) = \{\tilde \mu \in \mcM_1(X \times Y):\, \tilde{\mu} ( \cdot \times Y) = \mu_X \text{ and } \tilde{\mu}(X \times \cdot) =\mu_Y\}$. \par 
It is straight forward to generalize the above to finite metric measure spaces with $\overline{\mfx} = \overline{\mfy}$.

\begin{defn}\label{d.gen.Eur}
Let $\mfx,\mfy \in \bbM$, $\lambda > 0$.  The {\it generalized Eurandom metric} is defined as
\begin{equation}\label{eq.gen.Eurandom}
\dgEur^\lambda(\mfx,\mfy) := \inf_{\substack{\mfx',\mfy' \in \bbM,\ \overline{\mfx'} = \overline{\mfy'} \\ \mfx' \lemeasure \mfx, \ \mfy' \lemeasure  \mfy}} \left(D^\lambda(\mfx',\mfy'; \mfx,\mfy) + \dEur^\lambda(\mfx',\mfy') \right),
\end{equation}
where
\begin{equation}
\begin{split}
D^\lambda(\mfx',\mfy'; \mfx,\mfy) = \int (1 - &e^{-\lambda r}) \nu^{2, \mfx }(dr) -  \int (1 - e^{-\lambda r}) \nu^{2, \mfx'}(dr) \\
&+ \int (1 - e^{-\lambda r}) \nu^{2, \mfy}(dr) -  \int (1 - e^{-\lambda r}) \nu^{2, \mfy'}(dr).
\end{split}
\end{equation}
\end{defn}

Before we give the connection to $\legen$, we note that the generalized Eurandom distance is really a generalization of the 
Eurandom distance in the sense that $\overline{\mfx} = \overline{\mfy}$ implies $\dgEur^\lambda(\mfx,\mfy) = \dEur(\mfx,\mfy)$. 
Moreover one can prove that it metricizes the Gromov-weak topology on $\bbM$ (see \cite{MG} for details). \par 

%

We are now ready to give the main result of this section: 

\begin{thm}\label{t.gen.Eur.order}
Let $\mfx,\mfy \in \bbM$ with $\mfx \legen \mfy$, then
\begin{equation}
\dgEur^\lambda(\mfx,\mfy) =   \int (1- e^{-\lambda r} )\nu^{2,\mfy}(dr) - \int (1-e^{-\lambda r} ) \nu^{2,\mfx}(dr).
\end{equation}
\end{thm}

In order to prove this, we start by proving the analogue for the (non-generalized) Eurandom distance: 

 \begin{lem}\label{Eur.metric.dom}
Let $\mfx,\mfy \in \bbM_1$. Assume that $\mfx \lemetric \mfy$. Then the following holds: 
\begin{equation}
\dEur(\mfx,\mfy) = \int 1- e^{-\lambda r} d\nu^{2,\mfy} -  \int 1- e^{-\lambda r} d\nu^{2,\mfx}.
\end{equation}
\end{lem}

\begin{proof}
Let $\tau:\supp(\mu_Y) \rightarrow \supp(\mu_X)$ be a measure-preserving sub-isometry and define the measure $\tilde \mu$
on $\supp(\mu_X) \times \supp(\mu_Y)$ by setting $\tilde \mu(dx,dy) = \delta_{\tau(y)}(dx) \mu_Y(dy)$. Then $\tilde \mu$ is a coupling of $\mu_X$ and $\mu_Y$ and 
\begin{equation}
\begin{split}
\dEur^\lambda (\mfx,\mfy) &\le  \int |e^{-\lambda r_Y(y,y')}-e^{-\lambda r_X(x,x')}| \tilde \mu(d(x,y))\tilde  \mu(d(x',y')) \\
&= \int e^{-\lambda r_X(\tau(y),\tau(y'))} - e^{-\lambda r_Y(y,y')}  \tilde  \mu(d(x,y))\tilde \mu(d(x',y')) \\
&= \int 1- e^{-\lambda r} d\nu^{2,\mfy} -  \int 1- e^{-\lambda r} d\nu^{2,\mfx}
\end{split}
\end{equation} 
and ``$\le$'' follows. If $\tilde \mu$ is an arbitrary  coupling  of $\mu_X$ and $\mu_Y$, then
\begin{equation}\label{e.mg22}
\begin{split}
\int 1- e^{-\lambda r} \nu^{2,\mfy}&(dr) -  \int 1-e^{-\lambda r} \nu^{2,\mfx}(dr) \\
&\le \int |e^{-\lambda r_Y(y,y')}-e^{-\lambda r_X(x,x')}| \tilde \mu(d(x,y)) \tilde\mu(d(x',y')).
\end{split}
\end{equation}
\end{proof}

We are now ready to prove Theorem \ref{t.gen.Eur.order}:

\begin{proof}[Proof of Theorem \ref{t.gen.Eur.order}]
By definition there is a $\mfy' \in \bbM$ such that $\mfx \lemetric \mfy' \lemeasure \mfy$. First,
``$\le$'' follows if we choose $\mfx' = \mfx,\ \mfy' = \mfy'$ in the definition of $\dgEur^\lambda$ 
and apply Lemma~\ref{Eur.metric.dom} to this situation. \par 

For the ``$\ge$'' direction, let $\mfx'=[X',r_X',\mu_X']$, $\mfy' =[Y',r_Y',\mu_Y'] \in \bbM$, 
$\overline{\mfx'} = \overline{\mfy'}$ be minimizers of $\dgEur^\lambda(\mfx,\mfy)$. Such minimizers do always exist  (see \cite{MG}). 
By \eqref{e.mg22} we have 
\begin{equation}
\dEur^\lambda(\mfx',\mfy') \ge \left|\int 1- e^{-\lambda r} \nu^{2,\mfy'}(dr) -  \int 1-e^{-\lambda r} \nu^{2,\mfx'}(dr)\right| 
\end{equation}
If we set $f(r):= 1-e^{-\lambda r}$ and write $\nu^\mfx(f):= \int f d\nu^{2,\mfx}$, then this implies: 
\begin{equation}
 \begin{split}
\dgEur^\lambda(\mfx,\mfy) &\ge \nu^{\mfy}(f) +\nu^{\mfx}(f) - \nu^{\mfy'}(f) - \nu^{\mfx'}(f) + \left|\nu^{\mfy'}(f) - \nu^{\mfx'}(f) \right| \\
&= \nu^{\mfy}(f) +\nu^{\mfx}(f) - \nu^{\mfy'}(f) - \nu^{\mfx'}(f) \\
&{}\hspace{3cm}+   \nu^{\mfy'}(f)  + \nu^{\mfx'}(f)  - 2 \nu^{\mfx'}(f) \wedge\nu^{\mfy'}(f) \\
&=  \nu^{\mfy}(f) +\nu^{\mfx}(f)  - 2 \nu^{\mfx}(f) \wedge\nu^{\mfy}(f) \\
&= \nu^{\mfy}(f) +\nu^{\mfx}(f)  - 2 \nu^{\mfx}(f)  \\
&= \int (1- e^{-\lambda r}) \nu^{2,\mfy}(dr) -  \int (1- e^{-\lambda r}) \nu^{2,\mfx}(dr).
 \end{split}
\end{equation}
\end{proof}

\subsection{``Least upper bounds'' for \texorpdfstring{$\lemetric$}{le.dist}}\label{s.constr.lub}

We will now construct explicitly  the set of ``least upper bounds'' for $\lemetric$ using the properties of the Eurandom distance. Let $\mfx_1 = [X_1,r_1,\mu_1]$ and $\mfx_2 = [X_2,r_2,\mu_2]$ be both in $\bbM_1$. Consider an optimal coupling $Q := Q_{\mfx_1,\mfx_2}^\lambda \in \mcM_1(X_1\times X_2)$ s.t. the Eurandom distance 
\begin{equation}\label{eq.minim.Eurandom}
\dEur^{\lambda}(\mfx_1,\mfx_2) =  \int  \, |e^{-\lambda r_1(x_1,x_1')} - e^{-\lambda r_2(x_2,x_2')} |Q(\dx(x_1',x_2'))  Q(\dx(x_1,x_2)) 
\end{equation}
is minimized for a $\lambda > 0$. Such a coupling always exists (this is Lemma 1.7 in \cite{sturm} or alternatively Theorem 4.1 in \cite{Vil09}). We define

\begin{equation}
 \bar{r}((x_1,x_2),(x_1',x_2')) := r_1(x_1,x_1') \vee r_2(x_2,x_2'), \quad x_1,x_1' \in X_1,\, x_2,x_2' \in X_2
\end{equation}
and 
\begin{equation}
 \bar{\mfz} = [X_1 \times X_2, \bar{r}, Q] .
\end{equation}

\begin{prop}\label{prop.lub} Let $\mfx_1$, $\mfx_2,\bar{\mfz}$, $\lambda > 0$ be as above, then the following hold:
\begin{enumerate}
\item\label{i.lub.1} It is true that $\mfx_i \lemetric \bar{\mfz}$, $i=1,2$.
\item\label{i.lub.2} We have the following identity: 
\begin{equation}
\dEur^\lambda(\mfx_1,\mfx_2) = \dEur^\lambda(\mfx_1,\bar{\mfz}) + \dEur^\lambda(\bar{\mfz},\mfx_2).
\end{equation}
\item\label{i.lub.3} Let $\mfw = [X_3,r_3,\mu_3] \in \bbM_1$ with $ \mfx_i \lemetric \mfw$, $i=1,2$. If $\mfw \lemetric \bar{\mfz}$, then we have $\mfw = \bar{\mfz}$. 
\end{enumerate}
\end{prop}

\begin{proof} 

\eqref{i.lub.1} Consider the mapping $\pi_i : X_1 \times X_2 \to X_i$, $ (x_1,x_2) \mapsto x_i$, $i=1,2$. This mapping is measure-preserving on the correponding image set and a sub-isometry. 

\eqref{i.lub.2} We use Theorem \ref{t.gen.Eur.order} to calculate: 
\begin{equation}
 \begin{split}
\dEur^\lambda (\mfx_1,\mfx_2)  &= \int \left|e^{-\lambda r_1(x,y)} - e^{-\lambda r_2(x,y)} \right| Q(d(x,x')) Q(d(y,y')) \\
&= \int e^{-\lambda r_1(x,y)} + e^{-\lambda r_2(x,y)}  - 2 e^{-\lambda r_1(x,y)}\wedge e^{-\lambda r_2(x,y)} Q(d(x,x')) Q(d(y,y')) \\
&= \int e^{-\lambda r_1(x,y)} + e^{-\lambda r_2(x,y)} - 2 e^{-\lambda r_1(x,y) \vee r_2(x,y)} Q(d(x,x')) Q(d(y,y')) \\
&= \dEur^\lambda(\mfx_1,\bar{\mfz}) + \dEur^\lambda(\bar{\mfz},\mfx_2).
 \end{split}
\end{equation}

\eqref{i.lub.3} Note that Theorem \ref{t.gen.Eur.order} gives
\begin{align}
\dEur^\lambda (\bar{\mfz},\mfw) &= \dEur^\lambda (\mfx_1,\bar{\mfz})-\dEur^\lambda (\mfx_1,\mfw),\\
\dEur^\lambda (\bar{\mfz},\mfw) &= \dEur^\lambda (\mfx_2,\bar{\mfz})-\dEur^\lambda (\mfx_2,\mfw).
\end{align}
This implies 
\begin{equation}
\dEur^\lambda (\bar{\mfz},\mfw) = \frac{1}{2}\left(\dEur^\lambda (\mfx_1,\bar{\mfz}) + \dEur^\lambda (\mfx_2,\bar{\mfz})\right) 
-\frac{1}{2}\left(\dEur^\lambda (\mfx_1,\mfw) + \dEur^\lambda (\mfx_2,\mfw)\right) .
\end{equation}
Now we can use the result in \eqref{i.lub.2} and the triangle inequality to get 
\begin{equation}
\dEur^\lambda (\bar{\mfz},\mfw) \le  \frac{1}{2} \dEur^\lambda (\mfx_1,\mfx_2)  - \frac{1}{2} \dEur^\lambda (\mfx_1,\mfx_2) = 0.
\end{equation}
And therefore $\mfw = \bar{\mfz}$. 
%
\end{proof}

\section{Proofs of the main results}\label{s.proofs}



This section contains the proofs of Section~\ref{ss.legen}.

We start the proofs with a result which states that the definition of $\legen$ is a consequence of a similar statement where the roles of $\lemeasure$ and $\lemetric$ are reversed.

\begin{lem}\label{l.equi.def}
Let $\mfx = [X,r_X,\mu_X],\mfy= [Y,r_Y,\mu_Y],\mfx'= [X',r_X',\mu_X'] \in \bbM$ such that  $\mfx = [X',r_X',\mu_X]$$\lemeasure  [X',r_X',\mu_X'] $$\lemetric \mfy$. Then there is a Borel-measure $\mu_Y'$ on $Y$ such that $\mfx \lemetric [Y,r_Y,\mu_Y'] \lemeasure \mfy$.
\end{lem}

\begin{proof}
Let $\tau:\supp(\mu_Y) \to \supp(\mu_X')$ be a measure-preserving sub-isometry and take w.l.o.g. $Y = \supp(\mu_Y)$, $X' = \supp(\mu_X')$. We note that since $\mu_Y$ is tight, there is a sequence 
of compact sets $(K_n)_{n \in \mathbb N}$ such that $\mu_Y(K_n) \rightarrow \mu_Y(Y)$ and, since $\tau$ is measure-preserving:
\begin{equation}
\mu_Y(Y) = \mu_X'(X') \ge \mu_X'(\tau(K_n)) = \mu_Y(\tau^{-1}(\tau(K_n))) \ge \mu_Y(K_n),
\end{equation}
where we used the fact that $\tau(K_n)$ as the continuous image of a compact set is compact hence Borel. This implies $\mu_X'(\tau(K_n)) \rightarrow \mu_X'(X')$ and $\tilde \mu_X^{n}(A):= \mu_X'(A\cap\tau(K_n))\rightarrow \mu_X'(A)$ for all measurable $A \subset X'$. \par 
Fix a $n \in \mathbb N$ and recall that $\tau :K_n \to \tau(K_n)$ surjective Borel implies that the push-forward operator $\tau_\ast: \mathcal M_f(K_n) \to \mathcal M_f(\tau(K_n))$ is surjective Borel  (see  \cite{doberkat}, Proposition 1.101) and therefore we find a Borel-measure $\rho^n$ on $K_n$ (and hence on $Y$) such that $\rho^n \circ \tau^{-1} = \tilde \mu_X^n - \mu_X^n$, where $\mu_X^n := \mu_X(\cdot \cap \tau(K_n))$. Define
\begin{equation}
\nu_Y^n(A) := \mu_Y(A \cap K_n) - \rho^n(A),\qquad \forall A \in \sigma(\tau),
\end{equation}
where $\sigma(\tau) \subset \mathcal B(Y)$ is the sigma-field generated by $\tau$. We note that $\sigma(\tau)$ is countable generated 
and hence we can apply Lubin's Theorem (see \cite{lubin}) that gives an (not necessary unique) extension $\tilde \mu^n_Y$ of $\nu_Y^n$ to $\mathcal B(Y)$. Following the proof it is easy to see that $\tilde \mu^n_Y$ is a finite measure with $\tilde \mu^n_Y \le \mu_Y(\cdot \cap \tau(K_n))\le \mu_Y(\cdot)$ and in addition: 
\begin{equation}\label{eq.trans.2}
\begin{split}
\tilde \mu^n_Y(\tau^{-1}(A)) &= \mu_Y(\tau^{-1}(A) \cap K_n) - (\tilde \mu_X^n(A) - \mu_X^n(A)) \\
&\stackrel{n \rightarrow \infty}{\longrightarrow}  \mu_Y(\tau^{-1}(A)) - (\mu_X'(A) - \mu_X(A)) \\
&= \mu_X(A).
\end{split}
\end{equation}
Here we used that $\tilde \mu_X^n(A) \rightarrow \mu_X'(A)$ implies $\mu_X^n(A) \rightarrow \mu_X(A)$, since $\mu_X \le \mu_X'$. \par 
Finally observe that $\tilde \mu^n_Y \le \mu_Y$ implies relative compactness of $\{\tilde \mu_Y^{n}:\ n \in \mathbb N\}$ and if we take a limit point $\mu_Y'$ (along any subsequence) we get $\mu_Y'\le \mu_Y$. In addition by (\ref{eq.trans.2}) and the continuous mapping theorem, we find that $\mu_Y' \circ \tau^{-1} = \mu_X$.
\end{proof}

\begin{proof}[Proof of Theorem \ref{t.legen.pots}]
Reflexivity is clear and the  transitivity  is a consequence of  Remark \ref{r.equi.def} and Lemma \ref{l.equi.def}. \par 
For the antisymmetry observe that $\mfx \legen \mfy$ and $\mfy \legen \mfx$ implies that $\overline{\mfx} = \overline{\mfy}$ and hence we get 
$\mfx \lemetric \mfy$ and $\mfy \lemetric \mfx$. Since $\lemetric$ is a partial order, the result follows. \par 
Now let $\mfx_n,\mfy_n,\mfx,\mfy \in  \bbM$, $n \in \mathbb N$ with $\mfx_n \rightarrow \mfx$, $\mfy_n \rightarrow \mfy$ and $\mfx_n \legen \mfy_n$ 
for all $n \in \mathbb N$. By Remark \ref{r.equi.def} we find a sequence $(\mfy_n')_{n \in \mathbb N}$ in $\bbM$ with $\mfx_n \lemetric \mfy_n' \lemeasure \mfy_n$
for all $n \in \mathbb N$. By Proposition \ref{p.prop.measure} (\ref{p.lemeasure.compact}) we find $\mfy' \in \bbM$ such that $\mfy_n' \rightarrow \mfy'$ along some 
subsequence, where we suppress the dependence. Now, since both partial orders are closed, we get that $\mfx\lemetric \mfy'$ $\lemeasure \mfy$ and the result follows. 
\end{proof}

\begin{proof}[Proof of Proposition \ref{p.legen.compact}]
Let $L(A):= \bigcup_{\mfy \in A} \{\mfx \in \bbM:\, \mfx \legen \mfy\}$ and $(\mfx_n)_{n \in \mathbb N}$  be a sequence in $L(A)$. Then there is a sequence
$(\mfy_n)_{n \in \mathbb N}$ in $A$ and $(\mfy_n')_{n \in \mathbb N}$ in $\bbM$ such that $\mfx_n \lemetric \mfy_n' \lemeasure \mfy_n$ for all 
$n \in \mathbb N$ (see Remark \ref{r.equi.def}). Combining now Proposition \ref{p.prop.measure} (\ref{p.lemeasure.compact}) and Proposition \ref{p.prop.metric}
(\ref{p.lemetric.compact}) gives the result. 
 \end{proof}

We finally prove:

\begin{proof}[Proof of Proposition \ref{p.legen.easy}]
 Let $\mfx = [X,r_X,\mu_X]$ and $\mfy = [Y,r_Y,\mu_Y]$.
  We only verify the first statement for $a\in [0,1]$.
  Let $a\cdot \mfx = [X,r_X,a\mu_X]$.
  Use the mapping $\tau:X \to X$ with $\tau(x) = x$.
  Then $\tau$ is an isometry and $a \mu_X \circ \tau^{-1} = a \mu_X  \leq \mu_X$.
  All other statements may be verified similarly.
\end{proof}

\section{Applications}\label{s.applications}

We will now consider some applications for the partial orders.

\subsection{The Cartesian semigroup by Evans and Molchanov}\label{ss.EM14}

In \cite{EM14} a semigroup operation on $\bbM_1$ was introduced.
For $\mfx=[X,r_X,\mu_X], \mfy = [Y,r_Y,\mu_Y] \in \bbM_1$ they defined
\begin{equation}
 \mfx \boxplus \mfy = [X\times Y, r_X \oplus r_Y, \mu_X \otimes \mu_Y],
\end{equation}
where $r_X \oplus r_Y((x_1,y_1),(x_2,y_2)) = r_X(x_1,x_2) + r_Y(y_1,y_2)$ for $x_1,x_2 \in X$ and $y_1,y_2 \in X$.

Since the semigroup $(\bbM_1, \boxplus)$ is cancellative (see their Proposition~3.6) it is clear that there is also a partial order on $\bbM_1$ defined by
\begin{equation}
 \mfx \leq_{\boxplus} \mfy \ :\Leftrightarrow \ \exists z\in \bbM_1 \text{ s.t. } \mfx \boxplus \mfz = \mfy \,.
\end{equation}

This partial order is a special case of our order $\lemetric$ in the following sense.
\begin{prop}
 Let $\mfx, \mfy \in \bbM_1$ with $\mfx \leq_\boxplus \mfy$.
 Then $\mfx \lemetric \mfy$.
\end{prop}
\begin{proof}
 Let $\mfz \in \bbM_1$ such that $\mfx \boxplus \mfz = \mfy$.
 We may write $\mfy = [Y,r_Y,\mu_Y] = [X \times Z, r_X \oplus r_Z, \mu_X \otimes \mu_Z]$.
 Define the map $\tau: X \times Z \to X$ via $\tau(x,z) = x$.
 Then it is true that $\mu_Y \circ \tau^{-1} = \mu_X$, so $\tau$ is measure preserving and moreover for $y_1 =(x_1,z_1),y_2=(x_2,z_2) \in Y$
 \begin{equation}
  r_Y(y_1,y_2) = r_X(x_1,x_2) + r_Z(z_1,z_2) \geq r_X(x_1,x_2)  = r_X(\tau(y_1),\tau(y_2)) \, .
 \end{equation}
 Thus, $\mfx \lemetric \mfy$.
\end{proof}
An alternative proof via polynomials is the use of Lemma 3.2(b) in \cite{EM14} and Proposition \ref{p.lemetric.char} here.

\begin{rem}
 As Evans and Molchanov mention in the introduction they also could have chosen a different form of defining the metric $r_X \oplus r_Y$.
 They chose the $l_1$-addition, but also an $l_p$ addition of the form $r_X \oplus_p r_Y((x_1,y_1),(x_2,y_2)) = (r_X(x_1,x_2)^p + r_Y(y_1,y_2)^p)^{1/p}$ for $p \geq 1$ had led to a cancellative semigroup.
 For the order related to such a semigroup the previous proposition still holds.
\end{rem}

\subsection{General facts on stochastic dominance}\label{ss.stoch.dom}

Consider two random variables taking values in a partially ordered space $E$.
In which sense can the former be smaller than the latter? Even for $E = \R$ there are various concepts of a stochastic order.
We refer to the book of \cite{shaked2007stochastic} for a recent overview and collect some of the important results for us.

Let  $\mathcal X,\mathcal Y$ be  two  random variable with values in $\bbM$ and $\lambda > 0$. We define the Wasserstein distance (recall the definition of $\dgEur^\lambda$ in Section~\ref{s.gen.Eur}):
\begin{equation}
d_W^\lambda (\mathcal L(\mathcal X),\mathcal L(\mathcal Y)) := \inf_{Q} E_Q[\dgEur^\lambda (\mathcal X,\mathcal Y)],
\end{equation}
where the infimum is taken over all couplings of $\mathcal L(\mathcal X) $ and $\mathcal L(\mathcal Y) $.

\begin{rem}
Since $\dgEur^\lambda$ generates the Gromov-weak topology, convergence in $d_W^\lambda$ implies convergence in the weak topology on $\mathcal M_1(\bbM)$ (for all $\lambda > 0$). If we consider the space $\bbM_{\le K}$, i.e.\ mm-spaces with total mass bounded by some $K \ge 0$, then $\mathbb M_{\le K}$ is bounded (with respect to $\dgEur^\lambda$) and therefore $d_W^\lambda$ metricizes the weak topology on $\mathcal M_1(\bbM_{\mathbb K})$ (for all $\lambda > 0$) (see \cite{GS} for details). 
\end{rem}

\begin{defn}\label{d.stochastic.order}
 Let $(E,\prec)$ be a partially ordered set. For two random variables $\mathcal X$ and $\mathcal Y$ with values in $E$ we say that $\mathcal X \prec_{st} \mathcal Y$ ($\mathcal X$ is stochastically $\prec$-smaller) iff $\E[f(\mathcal X)] \leq \E[f(\mathcal Y)]$ for all bounded continuous increasing functions $f$.
\end{defn}

We recall the following result of Strassen \cite{Strassen}: 

\begin{prop}\label{p.Strassen.1}
Let $E$ be polish, $\prec$ be a closed partial order on $E$ and $\pi_1, \pi_2$ be two Borel probability measures on $E$. Then the following is equivalent:
\begin{enumerate}
\item There is  a Borel probability measure $\tilde \pi$ on $E \times E$, with marginals $\pi_1$ and $\pi_2$ such that $\tilde \pi(\{(x,y) \in E \times E: x \prec y\}) = 1$, 
\item For all real-valued bounded continuous increasing functions $f$ on $E$, $\int f d \pi_1 \le \int f d \pi_2$.
\end{enumerate}
\end{prop}

\begin{proof}
See \cite{Strassen} or \cite{L99}. 
\end{proof}

As a direct consequence of this proposition together with Theorem~\ref{t.gen.Eur.order}, we get:

\begin{prop}\label{c.Eurandom}
Let  $\mathcal X,\mathcal Y$ be  two  random variable with values in $\bbM$ and $\lambda > 0$. If  $\mathcal X \legen_{st} \mathcal Y$, then 
for all $\lambda > 0$:
\begin{equation}
d_W^\lambda (\mathcal X,\mathcal Y) = E\left[\int (1- e^{-\lambda r} )\nu^{2,\mathcal Y}(dr) \right] - E \left[\int (1-e^{-\lambda r} ) \nu^{2,\mathcal X}(dr)\right].
\end{equation}
\end{prop}

Let 
\begin{align}\label{e.tr163}
 \Pi_{\nearrow} := \{ \Phi \in \Pi \mid \Phi \text{ increasing}\} \qquad \text{and} \qquad  \Pi_{+,\nearrow} := \Pi_+ \cap \Pi_{\nearrow}.
\end{align}

Although it would be nice, we can not expect that increasing nonnegative polynomials $\Pi_{+,\nearrow}$ is enough to determine the stochastic order induced by $\legen$.  
This is not even true for polynomials in $\R$.
Nevertheless we may study the situation in which the stochastic order induced by $\Pi_{+,\nearrow}$ or $\Pi_{+}$ is just the right thing to look at.

\begin{defn}
 Let $(E,\prec)$ be a partially ordered set.
 For a cone $F \subset \{f: E \to \R \mid \text{increasing, bounded and measurable}\}$ define the stochastic order $\leq_\mcF$ on $\mcM_1(E)$ via
 \begin{equation}
  \mu \leq_\mcF \nu :\Leftrightarrow \int \mu(\dx x) \, f(x) \leq \int \nu(\dx x) \, f(x) \ \quad \forall f\in F.
 \end{equation}
 This definition extends to random variables in the obvious way.
\end{defn}

\begin{prop}
 The relations $\leq_{\Pi_{+,\nearrow}}$, $\leq_{\Pi_+}$  and $\leq_{\Pi_{\nearrow}}$ on $\mcM_1(\bbM)$  are partial orders.
\end{prop}

\begin{proof}
 We only provide the proof for $\leq_{\Pi_{+,\nearrow}}$.
 It is clear that $\leq_{\Pi_{+,\nearrow}}$ is transitive and reflexive.
 For anti-symmetry let $\mu, \nu \in \mcM_1(\bbM)$ with $\mu(\Phi) \leq \nu(\Phi) \leq \mu(\Phi)$ for all $\Phi \in \Pi_{+,\nearrow}$. 
 Then  $\mu(\Phi) = \nu(\Phi)$ for all $\Phi \in \Pi_{+,\nearrow}$ and thus this equality holds for all $\Phi$ in the algebra generated by $\Pi_{+,\nearrow}$.
 One may check that this algebra coincides with $\Pi$ and so Theorem 1 in \cite{GPW09} allows to deduce that $\mu = \nu$.
\end{proof}

We close this section with the following observation. For a random variable $\mathcal X \in \bbM_1$ we define the real-valued random variable $R_{12}^{\mathcal X}$ (on a different probability space) with law
\begin{equation}
 P(R_{12}^{\mathcal X} \in A) = \E[\nu^{2,{\mathcal X}}(A)], \ A \in \mcB([0,\infty)).
\end{equation}
$R_{12}$ models the random distance which we obtain by randomly picking two points the space.

\begin{prop}\label{p.polynom.order}
 Suppose ${\mathcal X} \leq_{\Pi_{+,\nearrow}} {\mathcal Y}$ for random variables ${\mathcal X}, {\mathcal Y} \in \bbM_1$. Then $R_{12}^{\mathcal X}\leq R_{12}^{\mathcal Y}$ stochastically.
\end{prop}

\begin{rem}
 This means that the stochastic order induced by $\Pi_{+,\nearrow}$ allows to state dominance of the (expected) sampled distance between two chosen individuals.
\end{rem}

\begin{proof}
By definition, ${\mathcal X} \leq_{\Pi_{+,\nearrow}} {\mathcal Y}$ implies 
\begin{equation}
E[ \phi(  R_{12}^\mfX) ] \leq E [ \phi(R_{12}^\mfY) ]  \ \text{ for all increasing } \phi \ge 0.
\end{equation}
By Proposition \ref{p.Strassen.1} the result follows. 
\end{proof}

Of course in the previous proof it had sufficed only to know things for the second order increasing monomials.

\subsection{Random graphs}

Consider the Erd\"os-Renyi random graph with parameters $(n,p)$, $n \in \N$ and $p \in [0,1]$.
That is the random graph consisting of $n$ vertices and a random collection of the possible $\binom{n}{2}$ edges between these points; edges are undirected.
Each of the possible edges is present with probability $p$ and is not present with probability $1-p$ and those choices are made independently of the other edges.
One possible way to construct such an object is to have $\binom{n}{2}$ independent Bernoulli($p$)-variables $(X_{ij})_{i<j \in E_n}$ if $E_n= \{1,\dotsc,n\}$ is the vertex set of the graph.
If $X_{ij}=1$, then the edge between vertices $i$ and $j$ is present, otherwise it is not present.

Define the random metric measure space
\begin{equation}
 \ER(n,p) = \left[E_n, r_n, n^{-1} \sum_{i\in E_n} \delta_i \right] \, ,
\end{equation}
where $r_n$ is the minimal graph distance of the random graph with the convention that $r_n(i,j) := n$ if $i$ and $j$ are not connected by a path.

Then we may establish the following result.
\begin{thm}
 For $p>p'$ and $n\in \N$ it is true that
 \begin{equation}
  \ER(n,p) \lemetric \ER(n,p') .
 \end{equation}
Moreover, the process $(\ER(n,p))_{p\in [0,1]}$ is an increasing Markov process taking values in $\bbM_1$.
\end{thm}
The proof can be obtained via coupling of the $X_{ij}$; we leave it out.

\subsection{Feller diffusion with drift}\label{s.feller}

The tree-valued Feller diffusion is the ultra-metric measure space valued process related to the Feller diffusion.
It can be seen as a many particle limit of Galton-Watson processes.
It is presented in \cite{ggr_tvF14} which considers the process $\mfU^{a,b} = (\mfU^{a,b}_t)_{t\geq 0}$ taking values in ultrametric measure spaces, denoted by $\mathbb U$; it is related to the total mass process $(X^{a,b}_t)_{t\geq 0}$ which solves the SDE $\dx X_t = bX_t \dx t + \sqrt{aX_t}\dx B_t$.
Here $a>0$ is the diffusivity and $b \in \R$ is the criticality of the offspring distribution.
The infinitesimal generator of $\mfU^{a,b}$ is given as in \cite{ggr_tvF14}:
\begin{align}
 L^{a,b} \Phi^{m,\phi}(\mfu) = \Phi^{m,2 \bar{\nabla}\phi}(\mfu) + bm\Phi^{m,\phi} + \frac{a}{\bar{\mfu}} \sum_{1\leq k \leq l \leq m} \Phi^{m,\phi\circ \theta_{k,l}}(\mfu) \, .
\end{align}
The notation for $\bar{\nabla} \phi = \sum_{1\leq k < l \leq m} \frac{\partial}{\partial{r_{kl}}} \phi$ and 
$\left( \theta_{k,l} (\underline{\underline{r}}) \right)_{i,j}  := r_{i,j}\1_{\{i\neq l, j\neq l\}} + r_{k,j} \1_{\{i=l\}} + r_{i,k} \1_{\{j=l\}} $
is taken from \cite{gpw_mp}.

\smallskip

It is well-known that for the total mass process one may couple two processes with different criticality and same initial condition.
More precisely, when $a>0$ and $b_1 < b_2 \in \R$, then we may define $X^{a,b_1}$ and $X^{a,b_2}$ on a joint probability space such that almost surely for all 
$t\geq 0$ we have $X_t^{a,b_1} \leq X_t^{a,b_2}$.
One way to prove that result is the classical comparison theorem for SDEs.

The following analogue for the tree-valued Feller diffusion holds true.
\begin{prop}\label{p.GaltonWatson}
 Let $\mfu \in \U$.
 For $a>0$ and $b_1 < b_2 \in \R$ let $\mfU^{a,b_i}$ be a solution of the $(L^{a,b_i},\delta_\mfu)$-martingale problem, $i=1,2$.
 We have for all $t>0$ almost surely,
 \begin{equation}
  \mfU_t^{a,b_1} \lemeasure \mfU_t^{a,b_2} \, .
 \end{equation}
\end{prop}
This result tells that the tree for $\mfU^{a,b_1}_t$ is really a subtree of $\mfU^{a,b_2}_t$ for any $t\geq 0$.

\begin{proof}
 Recall from \cite{Gl12} that there are Galton-Watson processes such that rescaling them leads to the processes $\mfU^{a,b_1}$ and $\mfU^{a,b_2}$.
 For example one may choose offspring distribution $\Poiss(1+b_1/N)$ and $\Poiss(1+b_2/N)$, respectively.
 It is well-known that $\Poiss(1+b_1/N) \leq \Poiss(1+b_2/N)$ stochastically, so we may couple the two processes such that the offspring distribution of the 
 $b_1$ process is always at most that of the $b_2$ process.
 Now, Proposition 3 in \cite{kamae1977} tells us that this coupling persists in the limit.
\end{proof}

\begin{rem}
 Of course the drift term $bX_t \dx t$ which appears in the last proposition may be changed to more general terms.
 For example one may also compare a process with linear drift and that with an additional quadratic death rate.
 This process is known as the logistic Feller diffusion.
 The same proof strategy allows to show that the process with the quadratic death rate can be coupled and be embedded in the process without that rate.
 This tells us that the genealogy of the logistic Feller diffusion can really be obtained by leaving out some individuals in the genealogical tree of the population without the death rate.
 This is suggested in \cite{le2013trees}.
 The right way to do remove individuals in a symmetric model, however, still remains unclear.
\end{rem}

Besides the proof of Proposition \ref{p.GaltonWatson} there is more indication for the result to hold.
In \cite{ruschendorf2008comparison}'s Remark~2.3~(b) it is mentioned that for
solutions of martingale problems in
partially ordered state spaces there is a generator criterion to deduce stochastic order with respect to a cone of functions.
R\"uschendorf provides a generator criterion for the cone $F$ of increasing functions (in our case $F = \Pi_+$) which allows to deduce $\leq_F$ stochastic dominance.
Even though $\leq_F$ is weaker than stochastic $\lemeasure$ dominance,  we find it instructive to present the easy calculation for the generator:
\begin{align} L^{a,b_1} \Phi^{m,\phi}(\mfu) & = L^{a,b_2} \Phi^{m,\phi}(\mfu)  + (b_1-b_2)m\Phi^{m,\phi}  \leq L^{a,b_2} \Phi^{m,\phi}(\mfu) ,
\end{align}
for all $\Phi^{m,\phi} \in \Pi$ with $\phi \geq 0$.

\subsection{Tree-valued Moran models}\label{ss.TVMM}

In this section we will prove a comparison result for two neural Moran models with different resampling rates. 
The proof depends on a comparison result of two Kingman-coalescents with different coalescing rates. Even though this 
comparison result is not new (on the coalescing level) it is new in the tree-valued setting.

We start with the (graphical) construction of the tree-valued Moran model as in \cite{gpw_mp}. Let $I_N:= \{1,\ldots,N\},\ N \in \mathbb N$ and
\begin{equation}\label{ppp}
\{\eta^{i,j}:\ i,j \in I_N,\ i \not=j\}
\end{equation}
be a realization of a family of independent rate $\gamma$ Poisson point processes, where we call $\gamma > 0$ the resampling rate.
If $\eta^{i,j}(\{t\}) = 1$, we draw an arrow from $(i,t)$ to $(j,t)$.

\begin{itemize}
\item[] For $i,i' \in I_N$, $0\le s< t < \infty$ we say that there is a path from $(i,s)$ to $(i',t)$ if there is a  $n \in \N$, 
$s \le u_1 < u_2 < \cdots < u_n \le t$ and
$j_1,\ldots,j_n \in I_N$ such that for all $k \in \{1,\ldots,n+1\}$ ($j_0:= i, j_{n+1}:=i'$)
$\eta^{j_{k-1},j_k}\{u_k\} = 1$, $\eta^{x,j_{k-1}}((u_{k-1},u_k))= 0$ for all $x \in I_N$.
\end{itemize}

\noindent Note that for all $i \in I_N$ and $0 \le s \le t$ there exists an unique element
\begin{equation}
A_s(i,t) \in I_N
\end{equation}
with the property that there is a path from $(A_s(i,t),s)$ to $(i,t)$. We call $A_s(i,t)$ the {\it ancestor of  $(i,t)$ at time $s$}. \par

Let $r_0$ be a pseudo-ultrametric on $I_N$. We define the pseudo-ultrametric ($i,j \in I_N$):
\begin{equation}\label{eq.ultrametric}
r_t(i,j):=\left\{ \begin{array}{ll}
2(t-\sup\{s \in [0,t]:A_s(i,t)=A_s(j,t)\}),&\quad \textrm{if } A_0(i,t) = A_0(j,t),\\[0.2cm]
2t + r_0(A_0(i,t), A_0(j,t)) ,&\quad \textrm{if } A_0(i,t) \not= A_0(j,t).
\end{array}\right.
\end{equation}

Define  $\mu^N\in \mathcal M_1(I_N)$ by
\begin{equation}
\mu^N= \frac{1}{N} \sum_{k \in I_N}  \delta_k. 
\end{equation}
Now, since $r_t$ is only a pseudo-metric, we  consider the following equivalence relation $\approx_t$  on $I_N$: 
 $x \approx_t y \Leftrightarrow r_t(x,y) = 0$.
We denote by $\tilde I_N^t:= I_N\! /\!\!\approx_t$ the set of equivalence classes and note that we can find a set of representatives $\bar I_N^t$
such that $\bar I_N^t \rightarrow \tilde I_N^t,\  x \to [x]_{\approx_t}$ is a bijection. We define 
\begin{align}
\bar r_t(\bar  i,\bar  j) =  r_t(\bar i,\bar j),\quad \bar \mu^N_t(\{\bar i\}) &=\mu^N(\{[\bar i]_{\approx_t}\}),\qquad  \bar i,\bar j \in \bar I_N^t.
\end{align} 

Then the {\it tree-valued Moran model (TVMM)}, of size $N$ is defined as
\begin{equation}
\mathcal U_t^{\gamma,N} := [\bar I_N^t,\bar r_t,\bar \mu^N_t]. 
\end{equation}

For the proof of the result below, it is better to construct the Moran model in a slightly different way:

\begin{rem}\label{r.alt.Def} Instead of a familiy of Poisson point processes as in (\ref{ppp}) we can also use a 
 rate $\gamma \cdot N(N-1)$  Poisson point process $\eta^\gamma $ and an i.i.d. sequence $(U_n)_{n \in \N} = (U_n^1,U_n^2)_{n \in \N}$ of $I_N \times I_N$-valued random variables with 
\begin{equation}
P(U_1 = (i,j)) = \frac{1}{N \cdot (N-1)} 1(i \neq j).
\end{equation}
We assume  that both are defined on the same probability space and are independent and set
\begin{equation}
\tau_k^\gamma = \inf\{t > \tau^\gamma_{k-1}:\ \eta^\gamma(\{t\}) = 1\},\qquad k \in \N. 
\end{equation}
Then we can construct the tree-valued Moran model as follows: At times $\tau_k^\gamma = t$ we draw an arrow from  $U^1_k = i$ to $U^2_k = j$, i.e.\ we sample two individuals $(i,j)$ independent and uniformly  without replacement of the population $I_N$  and then draw an arrow from $(i,t)$ to $(j,t)$ (see figure \ref{f.graph}). 
\end{rem}

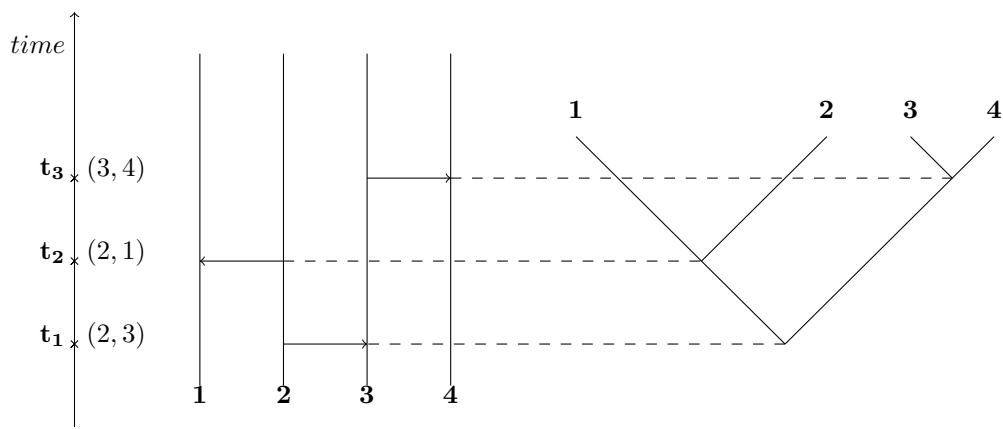
\begin{figure}[ht]
\centering
\begin{tikzpicture}[scale = 0.55]
\draw (1.,4.)-- (1.,-4.);
\draw (3.,4.)-- (3.,-4.);
\draw (5.,4.)-- (5.,-4.);
\draw (7.,4.)-- (7.,-4.);
\draw [->] (-2.,-5.) -- (-2.,5.);
\draw (1.0031720553104262,-3.7817153291833736) node[anchor=north] {$\mathbf{1}$};
\draw (3.0079033207116694,-3.7817153291833736) node[anchor=north] {$\mathbf{2}$};
\draw (4.994893601463344,-3.7817153291833736) node[anchor=north] {$\mathbf{3}$};
\draw (6.999624866864588,-3.7817153291833736) node[anchor=north] {$\mathbf{4}$};
\draw (10.,2.)-- (15.,-3.);
\draw (15.,-3.)-- (20.,2.);
\draw (19.,1.)-- (18.,2.);
\draw (13.,-1.)-- (16.,2.);
\draw (9.997851272641668,2.214737482370784) node[anchor=south] {$\mathbf{1}$};
\draw (15.99430408419583,2.214737482370784) node[anchor=south] {$\mathbf{2}$};
\draw (17.999035349597072,2.214737482370784) node[anchor=south] {$\mathbf{3}$};
\draw (20.003766614998316,2.214737482370784) node[anchor=south] {$\mathbf{4}$};
\draw [dash pattern=on 4pt off 4pt] (7.,1.)-- (19.,1.);
\draw [dash pattern=on 4pt off 4pt] (15.,-3.)-- (5.,-3.);
\draw [->] (5.,1.) -- (7.,1.);
\draw (-1.9950543504666547,-2.788220188807537) node[anchor=east] {$\mathbf{t_1}$};
\draw (-1.9950543504666547,-0.7834889234062947) node[anchor=east] {$\mathbf{t_2}$};
\draw (-1.9950543504666547,1.2212423419949474) node[anchor=east] {$\mathbf{t_3}$};
\draw (-1.9950543504666547,4.2194687477720265) node[anchor=east] {${time}$};
\draw [->] (3.,-1.) -- (1.,-1.);
\draw [->] (3.,-3.) -- (5.,-3.);
\draw [dash pattern=on 4pt off 4pt] (3.,-1.)-- (13.,-1.);
\draw (-1.0015592100908173,1.2212423419949474) node {${(3,4)}$};
\draw (-1.0015592100908173,-0.7834889234062947) node {${(2,1)}$};
\draw (-1.0015592100908173,-2.788220188807537) node {${(2,3)}$};
\begin{scriptsize}
\draw [color=black] (-2.,1.)-- ++(-2.5pt,-2.5pt) -- ++(5.0pt,5.0pt) ++(-5.0pt,0) -- ++(5.0pt,-5.0pt);
\draw[color=black] (-1.8708674579196751,1.3276882498923586) node {};
\draw [color=black] (-2.,-1.)-- ++(-2.5pt,-2.5pt) -- ++(5.0pt,5.0pt) ++(-5.0pt,0) -- ++(5.0pt,-5.0pt);
\draw[color=black] (-1.8708674579196751,-0.6770430155088837) node {};
\draw [color=black] (-2.,-3.)-- ++(-2.5pt,-2.5pt) -- ++(5.0pt,5.0pt) ++(-5.0pt,0) -- ++(5.0pt,-5.0pt);
\draw[color=black] (-1.8708674579196751,-2.6817742809101257) node {};
\end{scriptsize}
\end{tikzpicture}
\caption{\label{f.graph} \footnotesize Graphical construction of the TVMM (i.e.\ the tree on the right side): At times $t_1, t_2, t_3$ we sample to individuals $(x_1^1,x_2^1)$, $(x_1^2,x_2^2)$, $(x_1^3,x_2^3)$ and draw an arrow from $x_1^j$ to $x_2^j$.}
\end{figure}


In the following we will assume that $r_0 \equiv 0$, i.e.\ we start the process in $[\{1\},0,\delta_1]$.

\begin{prop}\label{prop.neutral}
Let $0 \le \gamma,\gamma'$. For all $N \in \N$ and $t \ge 0$, there is a coupling such that
\begin{equation}
P(\mathcal U^{\gamma + \gamma',N}_t\lemetric \mathcal U^{\gamma,N}_t ) = 1.
\end{equation} 
\end{prop}

\begin{proof}
We will only sketch the proof. For details about Kingman-coalescents see for example \cite{Ber}. 
\begin{proofsteps}
\step In this step we give the connection of $\mathcal U^{\gamma,N}$ and a Kingman $N$-coalescent with coalescing rate $\gamma$. \par  
For fixed $t \ge 0$, we set $A_h(i):=A_{t-h}(i,t)$, $0 \le h \le t$ and $[N]:= \{1,\ldots,N\}$. Then $\{A_h(i):\ i \in [N]\}$ can be described as a family of processes in $[N]^N$ that starts in $A_0(i) = i$ and has the following dynamic: Whenever $\eta^\gamma(\{t-h\}) = 1$ (see Remark~\ref{r.alt.Def}), we pick independent
and uniformly without replacement two individuals $i\neq j$ and have the following transition: 
\begin{equation}
A_{h-}(k) \rightarrow  A_{h}(k) = i, \qquad \forall k \in \{l \in [n]: A_{h-}(l) = j\}.
\end{equation}
It is now straightforward to see that the time it takes to decrease the number of different labels by $1$, given there are $k$ different labels, is exponential distributed with parameter $\gamma\cdot \binom{k}{2}$ and that the two labels (the one that replaces and the one that is replaced) are sampled uniformly without replacement under all existing labels. If we define  
\begin{equation}
\kappa_i(h) = \{j\in [N]:\ A_{h}(j) = A_{h}(i)\},
\end{equation}
this implies $\kappa = (\{\kappa_1(h),\ldots,\kappa_N(h)\})_{0\le h \le t}$ is a Kingman $N$-coalescent (up to time $t$). 
If we know define 
\begin{equation}
\mathcal V_t^N = [\{1,\ldots,N\},r^\kappa_t,\frac{1}{N}\sum_{k = 1}^N \delta_k],
\end{equation}
where 
\begin{equation}
r^\kappa_t(i,j) = 2\inf\{h \ge 0:\ i,j \in \kappa_k(h) \textrm{ for some }k\} \wedge 2t.
\end{equation}
then the above implies $\mathcal L(\mathcal V_t^N) = \mathcal L(\mathcal U_t^{\gamma,N})$.
\step Let $\kappa^{\gamma,N}$ and $\kappa^{\gamma + \gamma',N}$ be two Kingman $N$-coalescents with coalescing rate 
$\gamma$ and $\gamma+\gamma'$. Then one can couple this processes such that the coalescing times $\tau_i'$, $i = 1,\ldots,N-1$ of  $\kappa^{\gamma+\gamma',N}$ are dominated by the times $\tau_i$, $i = 1,\ldots,N-1$ of  $\kappa^{\gamma,N}$, i.e.\  $\tau_i' \le \tau_{i}$ for all $i = 1,\ldots,N-1$ almost surely. In addition to this property it is also possible to get a coupling such that $\kappa^{\gamma,N}(\tau_i) = \kappa^{\gamma+\gamma',N}(\tau_i')$ for all $i = 1,\ldots,N-1$.
\step Using the two steps above, we get the result with the identity as measure-preserving sub-isometry. 
\end{proofsteps}
\end{proof}

\subsection{Tree-valued Fleming-Viot processes}\label{sec.CompFV}

Let $(\mathcal U^{\gamma,N}_t)_{t \ge 0}$ be the TVMM with $\mathcal U^{\gamma,N}_0 = [\{1\},0,\delta_1]$. In this situation it is known that the TVMM  converges for $N \rightarrow \infty$, where the limit $(\mathcal U^{^\gamma}_t)_{t \ge 0}$  can be characterized as a
solution of a well-posed martingale problem. $(\mathcal U^{\gamma}_t)_{t \ge 0}$ is called the tree-valued Fleming-Viot process, TVFV, (see \cite{gpw_mp} or \cite{DGP12} for Details). As a consequence of Proposition \ref{prop.neutral}, we get:

\begin{prop}\label{p.FV}
Let $0 < \gamma < \gamma'$ and $t \ge 0$. Then there is a law $\lambda^{\gamma',\gamma}$ on $\mathbb U \times \mathbb U$ with marginals $\mathcal U^{\gamma'}_t$ and $\mathcal U^{\gamma}_t$ such that 
\begin{equation}
\lambda^{\gamma,\gamma'}(\{(\mfx,\mfy):\ \mfx\lemetric \mfy\}) = 1,
\end{equation}
or, in other words, there is a coupling such that $\mathcal U^{\gamma'}_t \lemetric  \mathcal U^{\gamma}_t$ almost surely. 
\end{prop}

\begin{proof}
This follows by Proposition~\ref{prop.neutral}  together with Proposition~\ref{p.lemetric.pots} and Proposition~\ref{p.prop.metric}, \ref{p.lemetric.compact} (see also Proposition~3 in \cite{kamae1977}).
\end{proof}

If we define, for $t \ge 0$, $R^\gamma_t$ as the distance of two randomly chosen points from $\mathcal U^{\gamma}_t$, i.e. 
\begin{equation}
P(R^\gamma_t \in A) = E[\nu^{2,\mathcal U^{\gamma}_t}(A)],
\end{equation}
for $A \subset \mathbb R_+$ measurable, then we get as a consequence (see Proposition~\ref{p.polynom.order}): 

\begin{cor}
For all $t \ge 0$, and $0 < \gamma < \gamma'$, there is a coupling such that 
\begin{equation}
P(R^{\gamma'}_t \le R^{\gamma}_t) = 1.
\end{equation}
\end{cor}

Another interesting observation is the following: By  Theorem 3 in \cite{gpw_mp}, there is a unique invariant law $\mathcal U^{\gamma}_\infty$ for the TVFV. Combining  Proposition~\ref{p.FV}  with Proposition~\ref{p.lemetric.pots} and Proposition~\ref{p.prop.metric}, \ref{p.lemetric.compact} shows that the result in Proposition~\ref{p.FV} stays true if we replace $t $ by $\infty$. If we now apply Proposition~\ref{c.Eurandom}, we get: 

\begin{prop} Let $0 < \gamma < \gamma'$, then for all $\lambda > 0$:
\begin{equation}
d_W^{\lambda}(\mathcal U^{\gamma}_\infty,\mathcal U^{\gamma'}_\infty)  = \frac{\gamma'}{\gamma'+\lambda} - \frac{\gamma}{\gamma + \lambda}.
\end{equation}
\end{prop}

\begin{proof}
Since $R^\gamma_\infty$ is $\Exp(\gamma)$ distributed (see for example Remark~3.16 in \cite{DGP12}) 
the result follows directly from the above discussion. 
\end{proof}


\bibliography{order}
\bibliographystyle{apt}

\end{document}